\documentclass[12pt]{amsart}
\usepackage{amscd}

\newcommand{\bP}{{\mathbb P}}
\newcommand{\bZ}{{\mathbb Z}}
\newcommand{\bF}{{\mathbb F}}
\newcommand{\bA}{{\mathbb A}}
\newcommand{\bB}{{\mathbb B}}
\newcommand{\bC}{{\mathbb C}}
\newcommand{\bD}{{\mathbb D}}

\newcommand{\la}{{\langle}}
\newcommand{\ra}{{\rangle}}

\newtheorem{thm}{Theorem}[section]
\newtheorem{lemma}[thm]{Lemma}

\newtheorem{cor}[thm]{Corollary}
\newtheorem{prop}[thm]{Proposition}

\numberwithin{equation}{section}

\title[   $p$-groups $rank_p=2$]
{The splitting of  cohomology of $p$-groups with rank $2$}
 
\author[A.Hida and N.Yagita]{Akihiko  Hida and Nobuaki Yagita}

\begin{document}

\address{[A. Hida] Department of Mathematics, Faculty of Education, Saitama University,
 Saitama, Japan}
\email{ahida@mail.saitama-u.ac.jp}

\address{[N. Yagita]
Department of Mathematics, Faculty of Education, Ibaraki University,
Mito, Ibaraki, Japan}
 
\email{yagita@mx.ibaraki.ac.jp}
\keywords{stable splitting, cohomology, classifying spaces}
\subjclass[2010]{Primary 55P35, 57T25, 20C20 ; Seconary 
 55R35, 57T05}

\begin{abstract}
Let $p$ be an odd prime and $BP$ the classifying space of a $p$-group $P$
with $rank_p(P)=2$.  By using stable homotopy splitting of $BP$, we study the decomposition of
$H^{even}(P;\bZ)/p$  and $CH^*(BP)/p$.
\end{abstract}

\maketitle

\section{Introduction}

Let $P$ be a $p$-group and $BP$ its classifying space.
We study the stable splitting and splitting of cohomology
\[(*)\quad BP\cong X_1\vee ...\vee X_i,\quad\]
\[(**)\quad H^*(P)\cong H^*(X_1)\oplus ... \oplus H^*(X_i)\quad (for\ *>0)\]
where $X_i$ are irreducible spaces in the stable homotopy category.
If we can get the splitting $(**)$ of cohomology, then we can study  
$H^*(G)_{(p)}$ for all groups $G$ having Sylow $p$ subgroups
which are isomorphic to $P$.

When Mitchell and Priddy [Mi], [Mi-Pr] started this problem, the 
splitting $(*)$ was
got by using the splitting $(**)$ of the cohomology of $P$. Since $H^*(P)$ are 
quite complicated for odd primes nonabelian $p$-groups, 
the examples were mostly given for $p=2$.

But after the answer of the Segal conjecture by Carlsson,
the splittings $(*)$ are  given by only using modular representation theory
by Nishida [Ni], Benson-Feshbach [Be-Fe] and Martino-Priddy [Ma-Pr].
In fact, their theorems  say that such decomposition is  decided only by structures of  simple modules of the mod$(p)$ double Burnside algebra $A(P,P)$.
These theorems do not use splittings of cohomology,  and  results are given for 
all primes $p$.

In particular, Dietz and Dietz-Priddy [Di], [Di-Pr] gave the stable 
splitting $(*)$
for groups $P$ with $rank_p(P)=2$ for $p\ge 5$.  However it was not used splittings
$(**)$ of the cohomology $H^*(P)$,  and the cohomologies $H^*(X_i)$ were not given there.

In [Hi-Ya1,2], we give the cohomology  of $H^*(X_i)$ (and hence $(**)$) for $P=p_+^{1+2}$
the extraspecial $p$ group of order $p^3$ and exponent $p$.
Their cohomology $H^*(X_i)$ are very complicated but 
have rich structures,  in fact $p_+^{1+2}$
is a $p$-Sylow subgroup of many interesting groups, e.g., $GL_3(\bF_p)$ and many simple groups
e.g. $J_4$ for $p=3$. 

In this paper, we give the decomposition of
\[H^*(P)=H^*(P;\bZ)/(p,\sqrt{0})\qquad  (and\ \ H^{ev}(P)=H^{even}(P;\bZ)/p) \]
for other $rank_pP=2$ groups
for odd primes $p$.
It is important to compute the transfer map for 
 the double Burnside algebra $A(P,P)$ action on $H^*(P;\bZ)$.
In general, it is not a so easy problem 
to compute the transfer on $H^*(P;\bZ)$.
However we can always compute it on $H^*(P)$ 
from Quillen's theorem.
  
In most cases, $H^*(X_i)$ are easily got but seemed not to have so rich structure
as $p_+^{1+2}$,
because they are not $p$-Sylow subgroups of so interesting groups.
However, we hope that from our computations, it becomes more clear the relations among splittings of  $H^*(P)$ of groups $P$ with $rank_p(P)=2$.

In particular, we note that the irreducible components are
most $fine$ in those of $rank_p=2$ groups,
namely,  the cohomology $H^*(X_{i}(P))$ 
 can  be written by using the decomposition
$H^*(X_{k}(p_+^{1+2}))$ (Lemma 8.1, Theorem 8.2, Theorem 8.5).  
\begin{thm}
For $p\ge 5$, let $P$ be a non-abelian $p$-group of $rank_pP=2$, which is not a  metacyclic group.
Let  $X_i(P)$ be an irreducible component of $BP$.  Then we can write
\[ H^*(X_i(P))\cong  \oplus_{j\in J(i,P)}H^*(X_j(p_+^{1+2}))\]
for some index set  $J(i,P)$.
\end{thm}
%\[ H^*(X_i(P))\cong \begin{cases}
%\oplus_{j\in J(i,P)'}H^*(X_j(p_-^{1+2}))
%\cong \oplus_{j\in J(i,P)}H^*(X_j(p_+^{1+2}))/T_j\\
%\qquad \qquad if\ P\ is\ a\ metacyclic\ group\\
% \oplus_{j\in J(i,P)}H^*(X_j(p_+^{1+2}))\quad otherwise
%\end{cases}\]
%for some sets $J(i,P)'$, $J(i,P)$ and submodules 
%$T_j\subset H^*(X_j(p_+^{1+2}))$.

This paper is organized as follows.
In $\S 2$ we recall the definition and properties of the double Burnside algebra  and the stable splitting.
In $\S 3$, we note the relations between splittings when cohomology of groups are isomorphic.
In $\S 4-\S 7$,  we give the stable splitting of the cohomology of $H^*(P)$ for the elementary abelian
group $\bZ/p\times \bZ/p$, 
metacyclic groups, $C(r)$ groups (such that $C(3)=p_+^{1+2}$), 
and $G(r',e)$ groups which appeared in the  classification of $rank_p(P)=2$ groups for $p\ge 5$ respectively.
 However, some  parts in $\S 4,\S 6$ are still given in 
[Hi-Ya1].  In $\S 8$, we study the relation among groups studied in $\S 5-\S 7$.
In $\S 9$ we study the nilpotent ideal parts in $H^{ev}(P)=H^{even}(P;\bZ)/p$ for all groups in $\S 4-\S 7$.
%In $\S 9$, we recall in the cases $p=2$ which are well known %(but notations using the Dickson invariant  seems 
%different from other literatures).
In $\S 10$, we note  the relation to the Chow ring $CH^*(BP)/p$ and $H^{ev}(P)$, and note that
the Chow group of  the direct summand $X_i$ is represented by some motive of the classifying space $BP$.

\section{the double Burnside algebra and stable splitting}

Let us fix an odd prime $p$ and $k=\bF_p$.  
For finite groups $G_1$, $G_2$, let 
 $A_{\bZ}(G_1,G_2)$ be the double Burnside group defined by
the Grothendieck group generated by $(G_1,G_2)$-bisets.
Each element $\Phi$ in $A_{\bZ}(G_1,G_2)$ is generated by elements
$[Q,\phi]=(G_1\times_{(Q,\phi)}G_2)$
for some subgroup $Q\le G_1$ and a homomorphism
$\phi:Q\to G_2.$
In this paper, we use the notation
\[[Q,\phi]=\Phi :G_1\ge Q\stackrel{\phi}{\to}G_2.\]

For each element $\Phi=[Q,\phi]\in A_{\bZ}(G_1,G_2)$, we can define a map  from  $H^*(G_2;k)$ 
to $H^*(G_1;k)$ by
\[ x\cdot \Phi=x\cdot [Q,\phi]=Tr_Q^{G_1}\phi^*(x)
\quad for \ x\in H^*(G_2;k).\]

When $G_1=G_2$, the group $A_{\bZ}(G_1,G_2)$ has the natural ring structure,
and call it the (integral) double Burnside algebra. 
In particular, for a finite group $G$, 
we have an $A_{\bZ}(G,G)$-module structure
on $H^*(G;k)$ (and $H^*(G;\bZ)/p$).

Recall  Quillen's theorem
such that the restriction map
\[ H^*(G;k)\to \lim_{V}H^*(V;k)\]
is an F-isomorphism (i.e. the kernel and cokernel are nilpotent) where $V$ ranges
elementary abelian $p$-subgroups of $G$.
We easily see ([Hi-Ya1]) 
\begin{lemma}
Let $\sqrt{0}$ be the nilpotent ideal in $H^*(G;k)$
(or $H^*( G;\bZ)/p$).
Then $\sqrt{0}$ itself is an $A_{\bZ}(G,G)$-module.
\end{lemma}

 In this paper we first (in $\S 4-\S 8$) consider the cohomology
modulo nilpotent elements,
since it is not so complicated  from the above Quillen's theorem.
 Hence  we write
it simply
\[H^*(G)=H^*(G;\bZ)/(p,\sqrt{0}).\]
However we also compute $H^{even}(G;\bZ)/p$ in $\S 9$ below.

By the preceding lemma,  we see that $H^*(G)$ has the $A_{\bZ}(G,G)$-module structure.
(Here note that $A_{\bZ}(G,G)$ acts on unstable cohomology.)
For ease of notations and arguments, when there is an  $A_{\bZ}(G,G)$-filtration 
$F_1\subset ...\subset F_n\cong H^*(P)$ such that
\[ grH^*(P)=\oplus F_{i+1}/F_i\cong \oplus m_iM_i\quad for\ *>0\]
with simple  $A_{\bZ}(G,G)$-modules  $M_i$, we simply
write
\[H^*(G)\leftrightarrow \oplus m_iM_i.\]
Throughout this paper, we assume that degree $*>0$ (or we consider $H^*(-)$ as the reduced theory $\tilde H^*(-)$).
(We consider $H^*(G)$ as an element in $K_0(Mod(A_{\bZ}(G,G)))$.)
In this paper, $H^*(G)\cong A$ for an graded ring $A$ means an graded module isomorphism
otherwise stated, while (in most cases) it  is induced from the ring isomorphism $grH^*(G)\cong A$ for some filtrations of $H^*(G)$.

Let $BG=BG_p$ be the $p$-completion of the classifying space of $G$.
Recall that $\{BG,BG\}_p$ is the ($p$-completed) group generated 
by stable homotopy self maps.
It is well known from the Segal conjecture (Carlsson's theorem)
that this group is isomorphic to the double Burnside group $A_{\bZ}(G_1,G_2)^{\wedge}$
completed by the augmentation  ideal.

Since the transfer is represented as a stable homotopy map $Tr$,
an element $\Phi=[Q,\phi]\in A(G_1,G_2)$ is represented as  a map
$\Phi\in \{BG_1,BG_2\}_p$
\[ \Phi: BG_1 \stackrel{Tr}{\to}
BQ\stackrel{B\phi}{\to}
BG_2. \]
(Of course, the action for $x\in H^*(G_2)$ is given by 
$Tr_{Q}^{G_1}\phi^*(x)$ as stated.)

Let us write
\[ A(G_1,G_2)=A_{\bZ}(G_1,G_2)\otimes k\quad (k=\bZ/p).\]
Hereafter we consider the cases $G_i=P$ ; $p$-groups.
Given a primitive idempotents decomposition of the unity of $A(P,P)$
\[1=e_1+...+e_n,\]
we have an indecomposable stable splitting
\[BP\cong X_1\vee...\vee X_n\quad with\ e_iBP=X_i.\]
In this paper,  an isomorphism $X\cong Y$ for spaces means that
it is a stable homotopy equivalence.

Recall that
\[M_i=A(P,P)e_i/(rad(A(P,P)e_i)\]
is a simple $A(P,P)$-module where $rad(-)$ is the  Jacobson radical.
By Wedderburn's theorem, the above
decomposition is also written as
\[ BP\cong \vee_{j}(\vee_{k} X_{jk})=\vee_{j}m_{j}X_{j1}\quad where\ m_{j}=dim(M_j)\]
for $A(P,P)e_{jk}/rad(A(P,P)e_{jk})\cong M_j$.
Therefore the  stable splitting of $BP$ is completely determined by
the idempotent decomposition of  the unity in the double Burnside algebra $A(P,P)$.

 For a simple $A(P,P)$-module $M$, define a stable summand
$X(M)$ by
\[ e_M=\sum_{M_i\cong M}e_i,\quad X(M)=\vee _{M_{jk}\cong M}X_{jk}=e_MBP.\]
Here $X(M)$ is only defined in the stable homotopy category.
(So strictly, the cohomology ring $H^*(X(M))$ is not defined.) However
we can define $H^*(X(M))$ as a graded submodule of the
cohomology ring $H^*(P)$  by
\[H^*(X(M))=H^*(P)\cdot e_M \quad (=e_M^*H^*(P)\ stabely)\]
where we think $e_M\in A(P,P)$ (rather than $e_M\in
\{ BP,BP\}$).
\begin{lemma}  Given a simple $A(P,P)$-module $M$, the cohomology\\
$grH^*(X(M))$ is isomorphic to a sum of $M$, i.e., (for $*>0$)
\[ H^*(X(M))\leftrightarrow \oplus_{i=1} M[k_i],\qquad  0\le k_1\le...\le k_s\le...\]
where $[k_s]$ is the operation ascending degree $k_s$.
\end{lemma}
\begin{proof}
Let $M'$ be a simple $A(P,P)$-module such that $M'\not \cong M$.
Then
\[e_{M'}X(M_i)=e_{M'}e_{M}BP=pt.\]
Hence $e_{M'}:H^*(P)\to H^*(P)$ restricts
$e_{M'}|H^*(X(M))=0$ (we assumed $*>0$).  This means that $H^*(X(M))$ does not contain
$M'$ as a  summand.
\end{proof}

From Benson-Feshbach [Be-Fe] and Martino-Priddy [Ma-Pr], it is known that each simple 
$A(P,P)$-module
is written as
\[ S(P,Q,V)\quad for \ Q\le P,\ and \ V: simple \ k[Out(Q)]-module.\]
(In fact $S(P,Q,V)$ is simple or zero.) 
Moreover, we have  (see [Be-Fe])
an isomorphism 
\[S(P,Q,V)\cong V\cdot  A(P,Q)/J(Q,V)\]
for some $A(P,P)$-submodule $J(Q,V)$ of $V\cdot A(P,Q)$.

Thus we have the main theorem of stable splitting of $BP$.
\begin{thm}(Benson-Feshbach [Be-Fe], Martino-Priddy [Ma-Pr])
There are indecomposable stable spaces $X_{S(P,Q,V)}$ for $S(P,Q,V)\not =0$ such that
\[BP\cong  \vee X(S(P,Q,V))\cong  \vee (dimS(P,Q,V))X_{S(P,Q,V)}.\]
\end{thm}

The direct summands $X_{S(P,P,V)}$ are called dominant summands ([Ni], [Ma-Pr]).  Let $X_{S=S(P,Q,V)}$ be a 
non-dominant summand for a proper subgroup $Q$.  Then it is known ([Ni], [Ma-Pr]) that the corresponding idempotent
$e_S\in A(P,P)$ is generated by elements  $P>Q\stackrel{\phi}{\to} P$ and
$P\to Q\to P$.  Hence
when there is no non-trivial map $P\to Q$, we see
\[ H^*(X_S)\cong e_SH^*(BP)\subset Tr_Q^PH^*(Q),\]
that is, $x\in H^*(X_S)$ if and only if
$\sum Tr_Q^P\phi^*(x)=x$.

\section{relation among groups $P$}

Let $R$ be a subring of $A(P,P)$.  For a simple $R$-module $S_R$, we can define
the idempotent $e_{S_R}$ and the stable space $Y_{S_R}=e_{S_R}BP$ which decomposes
$BP$, while it is (in general) not irreducible.  In particular, we take the group algebra of the outer automorphism group
$Out(P)$ as the ring $R$.  
\begin{lemma}
For each $Out(P)$-simple module $S_{R_i}$ with dimension $n_i$, 
let us write by 
\[ BP=n_1Y_1\vee... \vee n_sY_s\quad where \ Y_i=e_{S_{R_i}}BP\]
the decomposition for idempotents in $Out(P)$.
Then  each $Y_i$ decomposes
\[ Y_i=X_{i_1}\vee ...\vee X_{i_m} \quad for \ X_{ij}=e_{S_{ij}}BP\]
where $e_{S_{ij}}$ are idempotents  in $A(P,P)$. 
\end{lemma}

When $P$,$P'$ are different $p$ groups, the stable homotopy types
of $BP,BP'$ are different [Ni].  However there are many
cases with
$H^*(P)\cong H^*(P')$
 (However note that it seems not so often  that $H^*(P;\bZ)\cong  H^*(P';\bZ)$
even if we do not assume the map $P\to P'$
which induces the isomorphism on cohomology.)
The following corollary is immediate from the above lemma.

\begin{cor}
Let $P,P'$ are $p$-groups with $i_H:H^*(P)\cong H^*(P')$.
Assume that there is a ring map $i_A:A(P,P)\to  A(P',P')$
such that  $i_H(\Phi(x))=i_A(\Phi)i_H(x)$ for all $\Phi\in A(P,P)$
and $x \in H^*(P)$.  Then
for each splitting summand $X_i$ in $BP$, there are splitting summands  $X_{ij}'$
of $BP'$ such that
\[ i _H^*(X_i)=H^*(X_{i_{1}}')\oplus ...\oplus 
H^*(X_{i_{s}}').\]
\end{cor}
\begin{proof} We get the result from
\[i_HH^*(X_i)=i_H(e_iH^*(P))=i_A(e_i)i_HH^*(P)=i_A(e_i)H^*(P').\]
\end{proof}
\begin{prop}
Let $f:BP\to BG$ be a map such that
$f^*:H^*(G)\to H^*(P)$ is injective.  Let
$BP=\vee X_i(P)$ and $BG=\vee X_j(G)$ be the irreducible
decompositions.  For each $X_j(G)$, there are $i_1,..,i_s$
such that
\[ H^*(X_j(G))\cong f^*H^*(G)\cap(H^*(X_{i_1}(P)\vee ...\vee X_{i_s}(P)).\]
\end{prop}
\begin{proof}
We note each $f(X_i(P))$ is contained some $X_j(G)$
otherwise $X_i(P))$ should be decomposed.  
Let $f(X_{i_k}(P))\subset X_j(G)$. Then  we have a map
\[ f^*: H^*(X_j(G)) \to H^*(X_{i_1}(P)\vee.... \vee X_{i_s}(P)).\]
Since $f^*$ is injective, we have the proposition.
\end{proof}
If $G$ has Sylow $p$-group isomorphic to $P$,  then
of course the above proposition holds.  Moreover we consider
the cases that $P\subset G$ and $G$ is also a $p$-group
satisfying the above proposition in $\S 5$ below.

         \section{$A=\bZ/p\times \bZ/p$ for $p\ge 3$}

In this section, we recall  the decomposition of
the cohomology of $\bZ/p\times \bZ/p$, which  is still given 
 $\S 5$ in [Hi-Ya1].  However the result is not so trivial, and we write it briefly.  (The results used in the other sections are only Lemma 4.1 and Theorem 4.4.)

We recall the cohomology  
\[H^*(A)\cong k[u,y],\quad for \ A=\la a,b\ra \cong \bZ/p\times \bZ/p,\quad 
\]
where $u=c_1(e_a)$ is the first Chern class of a
non zero
linear representation $e_a:A\to \la a\ra\to \bC^{\times}$
and $y=c_1(e_b)$ is defined similarly.

At first we consider the case $B=\la b\ra\cong \bZ/p$
with $H^*(B)\cong k[y]$. The outer automorphism $Out(B)\cong \bF_p^{*}$ and its simple modules are written
$S_i=k\{y^i\}$ for $0\le i\le p-2$.
(Here we use the notation that $R\{x,y,...\}$ is the $R$-free
module generated by $x,y,...$.)
The summand  $L(1,i)$ is defined as $X(S(B,B,S_i))=X_{S(B,B,S_i)}$
and $H^*(L(1,i))\cong k[Y]\{y^i\}$ where $Y=y^{p-1}$.
 Hence we can decompose
\[ B\la b\ra \cong \vee _{i=0}^{p-2}L(1,i),\quad  H^*(L(1,i))\cong k[Y]\{y^i\}\quad with\ Y=y^{p-1}.\]

The outer automorphism $Out(A)\cong GL_2(k)$ and 
its simple modules are written as $S(A)^i\otimes det^j$
for $0\le i\le p-1, 0\le j\le p-2$. 
Here
\[ S^i(A)=H^i(A)\cong k\{y^i,y^{i-1}u,...,u^i\},\quad dim(S(A)^i)=i+1\]
and $det$ is the determinate  representation.
Let us write $L(1,i)=X_{S(P,\bZ/p,S_i)}$, that is the image of the same named component by the 
 split projections $A\to \la ab^{\lambda}\ra$, $\la b\ra$ $0\le \lambda\le p-1$.  Note that $Tr_A^P(x)=0$
for all $x$. Hence we can show  (e.g., Harris-Kuhn [Ha-Ku])  
\[BA\cong  \vee_{i,q}(i+1)\tilde X_{i,q}\vee _{i\not =0}(i+1)L(1,i)
\]
where $\tilde X_{i,q}=X_{S(A,A,S(A)^i\otimes det^q)}$
for  $0\le i\le p-1$, $0\le q\le p-2$, and $L(1,p-1)=L(1,0)$. 
However its decomposition of cohomology $H^*(A)$
is complicated.

For $j=(p-1)+i$ with $0\le i\le p-2$, we consider
\[S(A)^j=k\{y^j,\ y^{j-1}u,...,y^pu^{i-1},
T(A)^i,
y^{i-1}u^p,...,yu^{j-1},u^j\}\]
\[ with \quad  T(A)^i=k\{y^{p-1}u^i,\ y^{p-2}u^{i+1},\ 
...,\ y^iu^{p-1}\}.\] 
(Note $S(A)^{p-1}=T(A)^0$.)
Let $d_2=y^pu-yu^p\in H^*(A)^{SL_2(k)}$ so that
\[y^{j-1}u   =y^{i-1}u^p,\ ..., \ y^{p}u^{i-1}=yu^{j-1}\quad mod(d_2).\]
Then $S(A)^j\cong \tilde S(A)^i\oplus T(A)^i\  mod (d_2),$
 where
\[\tilde S(A)^j=k \{y^j\}\oplus k\{y^{j-1}u,...,,y^pu^{i-1}\}\oplus k\{u^j\}.\]

We can see $\tilde S(A)^j\ mod(d_2)$ is a $GL_2(k)$-module,
and hence $T(A)^i\cong S(A)^j/(\tilde S(A)^j, (d_2))$
is also $GL_2(k)$-module  (see [Hi-Ya1]).
Moreover, we have
\begin{lemma} (Lemma 4.3 in [Hi-Ya1])
There is an $Out(E)$-module isomorphism
$T(A)^i\cong S(A)^{p-1-i}\otimes det ^i[2i]$
 where $[2i]$ means the ascending degree $2i$ operation.
\end{lemma}
Hereafter  we use notation such that $A\ominus B=C$
means $A=B\oplus C$.
              \begin{thm}  We have an $Out(A)$-module decomposition
              \[H^*(A)\leftrightarrow k[d_2]\otimes( (k[\bar C]\otimes(\oplus_{i=0}^{p-2} S(A)^{i}
\oplus T(A)^i))\ominus k\{\bar C\})\]
  where $Out(A)$ acts trivially on $d_2$ and $\bar C$,
and $|\bar C|=2(p-1)$.             \end{thm}
{\bf Remark.}  The above theorem is proved in [Hi-Ya1],
by using the map   $ q : E\to E/\la c\ra \to A$
where $E=p^{1+2}_+$ and $\la c\ra$ is its center (see
$\S 6$). Then we can take 
        \[    grH^*(A)\cong Im(q^*)\oplus H^*(A)\{d_2\} 
           \cong Im(q^*)\otimes k[d_2].  \]
The right hand side above is the module in the theorem.
 There is an element $C\in H^{2(p-1)}(E)$ such that
$C\not \in Im(q^*)$ but $Cx\in Im(q^*)$ for 
$x\in H^+(A)$.  We define $\bar Cx=(q^*)^{-1}(Cx)$.
(Hence $\bar C$ itself does not exists in $grH^*(A)$.)
For example, $\bar Cy=Yy, \bar Cu=Uu$ with $U=u^{p-1}$, and 
$\bar C^2=Y^2+U^2-YU$, $\bar C^3=Y^3+U^3-Y^2U$
$mod(d_2)$.

   Let us write $\tilde H_{i,q}\cong (i+1)H^*(\tilde Y_{i,q}) $ is a  summand of $H^*(A)$ which is 
the  sum
         of all sub and quotient modules isomorphic to $S(A)^i\otimes det^q$.
         Let us write $D_2=d_2^{p-1}$. 
     Note that we use $k[\bar C]\ominus k\{\bar C\}\cong k[\bar C^2]\{1,\bar C^3\}$.

\begin{cor}  For $0\le q\le p-2$, we have 
\[ \tilde H_{i,q}\leftrightarrow \begin{cases}
k[\bar C^2,D_2]\{1,\bar C^3\}\quad  if\ i=q=0,\\ 
    k[\bar C^2,D_2]\{1,\bar C^3\}\{d_2^q\}\quad if \ i=0,\ q>0,\\
 k[\bar C,D_2]\otimes(S(A)^i\otimes d_2^q
   \oplus T(A)^{p-1-i}\otimes d_2^{i+q})\quad otherwise.
   \end{cases}\]
  \end{cor}

Since $H^*(L(1,i))\subset H^*(\tilde Y_{i,0})$, (in fact $H^*(L(1,i)$
does not contain $d_2^q$), we have
\[\tilde Y_{i,q}\cong \begin{cases} \tilde X_{i,q}\qquad    if\  q\not =0\ or\  (i,q)=(0,0)\\
                                              \tilde X_{i,0}\vee L(1,i)\qquad if\  q=0,\ i\not =0.
\end{cases}\]
Let us write $\tilde \bC\bB=k[\bar C,D_2]\cong
k[Y,D_2]$.  Then we get   
  \begin{thm}  ([Hi-Ya1])
We have $H^*(\tilde X_{0,q})\cong \tilde H_{0,q}$.
 For $i\ge 1$, we have 
\[(i+1)H^*(\tilde X_{i,q})\cong 
\begin{cases} \tilde {\bC}\bB\otimes (S(A)^i\{D_2\}\oplus 
   T(A)^{p-1-i}\otimes d_2^{i}),\quad q =0\\
    \tilde {\bC}\bB\otimes (S(A)^i\{d_2^q\}\oplus 
   T(A)^{p-1-i}\otimes d_2^{i+q})\quad q\not =0.
\end{cases}\]
   In particular, the space $\tilde X_{p-1,q}$ 
(which is written as $L(2,q)$ e.g., in [Mi-Pr]),   \[pH^*(L(2,q))\cong \begin{cases}
       \tilde {\bC}\bB\otimes (S(A)^{p-1}\{D_2\}),\quad q=0\\
        \tilde {\bC}\bB\otimes (S(A)^{p-1}\{d_2^q\}) \quad q\not =0.
\end{cases}\]
 \end{thm}
\begin{proof}
We only need to prove the case $i\not =0$ and $q=0$.
First note
\[ (i+1)H^*(\tilde X_{i,0})\cong \tilde H_{i,q}\ominus (i+1)H^*(L(1,2).\]
Using Corollary 4.3, we can compute
\[ (i+1)H^*(\tilde X_{i,0})\ominus (\tilde \bC\bB \otimes T(A)^{p-1-i}\otimes d_2^i)\]
\[ \cong \tilde \bC\bB\otimes (S(A)^i)\ominus (i+1)k[Y]\{y^i\}\cong (\tilde \bC\bB\ominus k[Y])\otimes S(A)^i.\]
Here we identify $(i+1)k\{y^i\}\cong k\{y^i,y^{i-1}u,...,u^i\}\cong S(A)^i$.  
The result follows from  
\[\tilde \bC\bB\ominus k[Y]\cong k[Y,D_2]\ominus k[Y] \cong k[Y,D_2]\{D_2\}\cong \tilde \bC\bB\{D_2\}.\]
\end{proof}

 We recall the Dickson algebra $\bD\bA$, namely,
\[ \bD\bA=k[y,u]^{GL_2(k)}=k[D_1,D_2]\]
where $D_1=Y^p+V$ and $V=D_2/Y $
 (see $\S 5,6$ below).
 Note that
 $\bar C^p=V+Y^p=D_1\ mod(d_2)$.  Hence
we can identify (as a free $\bD\bA$-module)
\[ \tilde \bC\bB=\bD\bA\{1,\bar C,...,\bar C^{p-1}\}.\] 
Note $H^*(\tilde  X_{0,0})\not \cong \bD\bA$.
In fact, we have
\begin{lemma}
\[H^*(\tilde X_{0,0})\cong k[\bar C^2,D_2]\{1,\bar C^3\} \cong
 \bD\bA\{1,D_1\bar C,\bar C^2,...,\bar C^{p-1} \}.\]
\end{lemma}
\begin{proof}  We see the last isomorphism
\[ \bC\bB\ominus k\{\bar C\}\otimes k[D_2]
\cong \bD\bA\{1,\bar C,..., \bar C^{p-1}\}
\ominus k\{\bar C\}\otimes k[D_2] \]
\[ \cong \bD\bA\{1,\bar C^2,...,C^{p-1}\}\oplus A\quad 
with \ A=\bD\bA\{\bar C\}\ominus k[D_2]\{\bar C\}.\]
Here $A\cong (\bD\bA\ominus k[D_2])
\{\bar C\}
\cong \bD\bA
\{D_1\bar C\}.$
Thus we have the result.
\end{proof}

     {\bf Examples.} 
 For $p=3$, (see Corollary 5.2 in [Ya2]) we have   
  \[H^*(\tilde X_{0,0})\cong \tilde H_{0,0} \cong H^*((A:SD_{16}))\cong H^*(A)^{SD_{16}}\]
\[ \cong \bD\bA\{1,\tilde CD_1,\tilde C^2\}\cong \bD\bA\{1,\tilde C^2,\tilde C^4\}.\]

\section{metacyclic groups for $p\ge 3$ }

For $p\ge 5$,  groups $P$ with $rank_pP=2$ are classified by Blackburn (see Thomas [Th], Dietz-Priddy [Di-Pr] ).
They are metacyclic groups, groups $C(r)$ and $G(r',e)$ (see sections 5,7 below for the definitions).
In this section, we consider metacyclic $p$ groups $P$ for $p\ge 3$
\[    0\to \bZ/p^m \to P\to \bZ/p^{n}\to 0.\]
These groups are represented as
\[ P=\la a,b|a^{p^m}=1, a^{p^{m'}}=b^{p^n},[a,b]=a^{rp^{\ell}}\ra \quad r\not=0\ mod(p)\quad (5.1). \]
It is known by Thomas [Th], Huebuschmann [Hu]  that $H^{even}(P;\bZ)$ is multiplicatively generated
by Chern classes of complex representations.  Let us write
\[\begin{cases} y=c_1(\rho),\quad  \rho: P\to P/\la a\ra \to \bC^*\\
 v=c_{p^{m-\ell}}(\eta),\quad \eta= Ind_H^P(\xi),\ \ 
\xi:H=\la a,b^{p^{m-\ell}}\ra \to \la a\ra \to \bC^*
\end{cases}\]
where $\rho,\xi$ are nonzero linear representations.
Then $H^{even}(P;\bZ)$ is generated by
\[ y,\ c_1(\eta),c_2(\eta),...,c_{p^{m-\ell}}(\eta)=v.\]
(Lemma 3.5 and the explanation just before this lemma in [Ya1].)
We can see (the last equation in the proof of Theorem 5.45 in [Ya1])
\[c_1(\eta)=0,...,c_{p^{m-\ell}-1}(\eta)=0 \quad  in\ H^*(P)=H^*(P;\bZ)/(p,\sqrt{0}).\]
By using Quillen's theorem and the fact  that $P$ has just one conjugacy class of maximal abelian
$p$-subgroups,  we can prove
\begin{thm} (Theorem 5.45 in [Ya1])
For any metacyclic $p$-group $P$ in (5.1) with $p\ge 3$, we have a ring  isomorphism
\[ H^*(P)\cong k[y,v],\quad |v|=2p^{m-\ell}\quad (5.2).\]
\end{thm}

We now consider the stable splitting.

(I) Non split cases.
For a  non split metacyclic groups, it is proved  that $BP$ itself  is irreducible [Di].

(II) Split cases with $(\ell,m,n)\not =(1,2,1)$.
We  consider  a split metacyclic group.  it is written as 
\[P=M(\ell,m,n)=\la a,b|a^{p^m}=b^{p^n}=1, [a,b]=a^{p^{\ell}}\ra\]
for $m>\ell\ge max(m-n,1)$.  

The outer automorphism is the semidirect product
\[Out(P)\cong (p-group):\bZ/(p-1).\]
The $p$-group acts trivially on $H^*(P)$,  and $j \in \bZ/(p-1)$
acts on $a\mapsto a^{j}$ and so acts on  $H^*(P)$ as $j^* : v\mapsto j v$.
There are $p-1$ simple $\bZ/(p-1)$-modules $S_i\cong k\{v^i\}$. We consider the decomposition
by idempotens for $Out(P)$.  Let us write $Y_i=e_{S_i}BP$ and
\[ H^*(Y(S_i))\cong (dim(S_i))H^*(Y_i)\subset H^*(P)\]
(in the notation $Y_i$ from Lemma 3.1). Hence we have the decomposition for $Out(P)$-idempotents
\[ H^*(Y_i)\cong k[y,V]\{v^i\},\quad V=v^{p-1}.\]

We assume $P\not =M(1,2,1)$.
By Dietz, we have splitting
\[(*)\quad  BP\cong \vee_{i=0}^{p=2}X_i\vee \vee_{i=0}^{p-2}\bar L(1,i).\]
Here we  write $X_i=e_{S(P,P,S_i)}BP$  identifying  $S_i$ as the $A(P,P)$ simple module
(but not the simple $Out(P)$-module). 

The summand  $\bar L(1,i)$ is defined as follows.
(When $n=1$, $\bar L(1,i)=L(1,i)$ defined in $\S 4$.)
Recall that $H^*(\la b\ra)\cong k[y]$.  The outer automorphism group is  
$Out(\la b\ra)\cong (\bZ/p^n)^*$ and 
its simple $k$ modules are $S_i'=k\{y^i\}$ for $0\le i\le p-2$.  Hence we can decompose
\[ B\la b\ra \cong \vee _{i=0}^{p-2}\bar L(1,i),\quad  H^*(\bar L(1,i))\cong k[Y]\{y^i\}\quad with\ Y=y^{p-1}.\]

Next we consider $\bar L(1.i)$ as a split summand in $BP$ as follows.
(Consider the $A(P,P)$-simple module $S(P,\la b\ra, S_i')$.)
Let $\Phi\in A(P,P)$ be the element  defined by the map
$ \Phi: P\ge P\to \la b\ra \subset P$
which induced the isomorphism
\[H^*(P)\Phi\cong H^*(\vee_{i=0}^{p-2} \bar L(1,i))\cong k[y]
\subset H^*(Y_0).\]
Thus we can show (since $k[y]$ is invariant under elements in $Out(P)$)
\[(**)\quad Y_i\cong \begin{cases} X_i\quad i\not =0\\
                                  X_0\vee \vee_{j=0}^{p-2}\bar L(1,j)\quad i=0.
\end{cases}\]

{\bf Remark.}   For groups $P,P'$ with the same $m-\ell$, we have the isomorphism
$H^*(P)\cong H^*(P')$ and the Burnside algebras act on the cohomology by the same way.
For the splittings $X(P)$ and $X(P')$ (for $BP$ and $BP'$ respectively), we have
$H^*(X_i(P))\cong H^*(X_i(P'))$.  But when $P\not \cong P'$, it is known from
Nishida [Ni] that $X_i(P)\not \cong X_i(P')$, i.e. they are not stably homotopy equivalent.
Similarly $\bar L(1,i)$ are  different stable homotopy types when $n$ are different.
%{\bf Remark.}  We also get $(*)$ from the above $(**)$ and
% we correct the number of $L(1,i)$ in Theorem 1.1 (2) in [Di].

\begin{thm}  Let $P$ be a split metacyclic group with $(\ell,m,n)\not =(1,2,1)$.
Then we have 
\[ H^*(X_i)\cong \begin{cases} k[y,V]\{v^i\} \quad i\not = 0\\
                         k[y,V]\{V\}\quad i=0.
       \end{cases}\]
\end{thm}
\begin{proof} For $i\not =0$, we have
         $H^*(Y_i)\cong H^*(X_i)$.
For $i=0$,  we see 
\[H^*(X_0)\cong  H^*(Y_0)\ominus  H^*(\vee _{j=0}^{p-2}L(1,j)) \]
\[ \cong k[y,V]\ominus k[y] \cong k[y,V]\{V\}.\]
\end{proof}

(III) Split metacyclic group with $(\ell,m,n)=(1,2,1).$

 This case $P=p_{-}^{1+2}$ and its cohomology is the same as (II).  But the splitting is given ([Di], [Di-Pr])
\[ BP\cong \vee_{i=0}^{p=2}X_i\vee \vee_{i=0}^{p-2}L(2,i)\vee \vee_{i=0}^{p-2}L(1,i).\]
Let $H=\la b,a^p\ra$ the maximal elementary abelian subgroup.
The outer automorphism $Out(H)\cong GL_2(k)$ and
simple $GL_2(k)$-modules are written as $S(H)^i\otimes det^j$.
The summand $L(2,i)$ is defined as 
\[  L(2,i)=X_{S(P,H,S(H)^{p-1}\otimes det^i)}.\]

Here note $v|_{H}=u^p-y^{p-1}u$ so that $yv|_{H}=d_2$.
This fact is proved by using the fact that
$v|_{H}$ invariant under the action $a^*:u\mapsto u+y,
y\mapsto y$.

Of course there is no map $P\to H$. 
The space $L(2,i)$ is the transfer ($Tr:BH\to BG$)
 image of the same named summand of $BH$.  

  By using the double coset formula
\[Tr_H^P(u^{p-1})|_H=\sum_{i=0}^{p-1} (u+iy)^{p-1}=-y^{p-1}\]
taking the generator $u$ in 
$H^*(\la b,a^p\ra)\cong k[y,u]$.
The group $P$ has just one conjugacy class $H$ of the maximal abelian $p$-groups.
Hence by Quillen's theorem, we have
\[ Tr_H^P(\bar C^id_2^ju^{p-1})=-Y^i(yv)^jY\quad in\ H^*(P)=H^*(P;\bZ)/(p,\sqrt{0}).\]

Next, we consider an element $\Phi\in A(P,P)$ defined by
\[\Phi:  P\ge H\stackrel{a^p\leftrightarrow b}{\cong} H \subset P.\]
Then $\Phi(C^id_2^j)=-C^{i}d_2^j$. (So $\Phi^2(C^id_2^j)=C^id_2^j)$.)
Here recall Theorem 4.4
\[ H^*(L(2,i))\cong \begin{cases}
                                         \tilde \bC\bB\{\bar Cd_2^i\} \quad i\not =0\\
\bar  \bC\bB\{\bar CD_2\} \quad i=0.
\end{cases}\]

Since 
$C^id_2^j=C^iy^jv^j\in H^*(Y_j)$, we see $\Phi(H^*(L(2,j))\subset H^*(Y_j)$.
Thus we have the isomorphism
\[Y_i\cong \begin{cases} X_i\vee L(2,i)\qquad i\not =0\\
          X_0\vee L(2,0)\vee \vee_{j=0}^{p-2}L(1,j)\qquad i=0.
\end{cases} \]

To compute cohomology of irreducible components $X_i$ and $L(2,j)$, we recall 
the Dickson algebra
\[\bD\bA=k[y,u]^{GL_2(\bZ/p)}\cong k[D_1,D_2]\quad with\ D_1=Y^p+V,\ D_2=YV.\]
We also write (see $\S 6 $ bellow) the free $\bD\bA$-modules
\[\bC\bA=k[Y,V]\cong \bD\bA\{1,Y,...,Y^p\},\]
\[\bC\bB=k[Y,D_2]\cong \bD\bA\{1,Y,...,Y^{p-1}\}.\]
Hence $\bC\bA\cong \bD\bA\oplus \bC\bB\{Y\}$.

\begin{thm}
Let $P=M(1,2,1)\cong p_-^{1+2}$.  Then we have
\[H^*(X_i)\cong \begin{cases} \bC\bA\{1,...,\hat y^{i},...,y^{p-2}\}
\{v^i\}\oplus \bD\bA
\{d_2^i\}\quad i>0\\
\bC\bA\{y,,...,y^{p-2}\}\{V\}\oplus \bD\bA\quad i=0.
\end{cases}\]
\end{thm}
\begin{proof}
Let $i\not =0$.  We see 
\[H^*(Y_i)\cong k[y,V]\{v^i\}\cong \bC\bA\{1,y,...,y^{p-2}\}\{v^i\}.\]
The cohomology of the summand $X_i$ is
\[H^*(X_i)\cong H^*(Y_i)\ominus H^*(L(2,i))\]
\[\cong \bC\bA\{v^i\}\{1,...,y^{p-2}\} \ominus \bC\bB\{Yd_2^i\}\]
\[\cong \bC\bA\{1,...,\hat y^i,...,y^{p-2}\} \{v^i\}
\oplus( \bC\bA\{v^iy^i\}\ominus\bC\bB\{Yd_2^i\}).\]
Here $v^iy^i=d_2^i$ and $\bC\bA\cong \bD\bA\oplus
\bC\bB\{Y\}$, and
we have the isomorphism in the theorem for $i\not =0$.

Next  we consider in the case $i=0$.  From Theorem 5.2, 
we see
\[ H^*(Y_0)\ominus H^*(\vee _jL(1,j)\cong
k[y,v]\{V\}\cong \bC\bA\{1,...,y^{p-2}\}\{V\}.\]
Hence we have 
\[H^*(X_0)\cong H^*(Y_0)\ominus H^*(\vee_jL(1,j))\ominus H^*(L(2,0))\]
\[\cong \bC\bA\{1,y,...,y^{p-2}\}\{V\}\ominus \bC\bB\{YD_2\}
 \cong \bC\bA\{y,...,y^{p-2}\}\{V\}\oplus B\]
where 
\[  B=\bC\bA\{V\}\ominus \bC\bB\{YD_2\}
 \cong \bC\bA\ominus H^*(L(1,0))\ominus H^*(L(2,0)). \]
We can see $B\cong \bD\bA$ by the following lemma.
\end{proof}
\begin{lemma}  Let $M(2)=L(2,0)\vee L(1,0)$ (as the usual notation of the homotopy theory).
Then we have 
\[ H^*(M(2))\cong \bC\bB\{Y\},  \quad
\bC\bA\cong \bD\bA\oplus H^*(M(2)).\]
\end{lemma}
\begin{proof}  We can compute
\[H^*(M(2))\cong k[Y]\oplus \bC\bB\{YD_2\}
\cong k[Y]\oplus k[Y,D_2]\{YD_2\} \]
\[\cong (k[Y]\oplus k[Y,D_2]\{D_2\})\{Y\}\cong \bC\bB\{Y\} \quad (assumed \ *>0).\]
Since $\bC\bA\cong \bD\bA\oplus \bC\bB\{Y\}$, we have the last isomorphism in this lemma.
\end{proof}

\section{$C(r)$ groups for $p\ge 3$}

The group $C(r),\ r\ge 3$ is the $p$-group of order $p^r$ such that 
\[ C(r)=\la a,b,a| a^p=b^p=c^{p^{r-2}}=1,\ [a,b]=c^{p^{r-3}}\ra\]
for $r\ge 3$ so that $C(3)=p_+^{1+2}$.  Hence we have a central extension
\[ 0\to \bZ/p^{r-2}\to C(r)\to \bZ/p\times \bZ/p \to 0.\]

For each $r\ge 3$, the  cohomology $H^*(C(r))$ is isomorphic to $H^*(C(3))$.
Denote $C(3)=p_+^{1+2}$ simply by $E$. 
The cohomology of  $E$ is well known.
In particular recall that  ([Lw], [Le1,2],[Te-Ya])
\[(1)\quad H^*(E)\cong (k[y_1,y_2]/(y_1^py_2-y_1y_2^p)\oplus k\{C\}
)\otimes k[v].\]
Here $y_1$ (resp. $y_2$) is the first Chern class
$c_1(e_1)$ (resp. $c_1(e_2)$) for the nonzero linear representation $e_1:E\to \la a\ra \to \bC^*$
(resp. $e_2:E\to \la b\ra \to \bC^*$).
 The  elements $C$ and $v$ are also represented by Chern classes
\[ c_i(Ind_{A}^E(e))=\begin{cases} v\quad for \ i=p\\
                 C\quad for \ i=p-1\\
                   \end{cases}  \]  
                    where $e:A\to \la c\ra \to \bC^*$
                    is a non zero linear representation,
for any maximal elementary abelian subgroup $A$.
Hence $|y_i|=2,|C|=2(p-1),|v|=2p$.
It is well known  $Cy_i=y_i^{p}$, 
$C^2=y_1^{2p-2}+y_2^{2p-2}-y_1^{p-1}y_2^{p-1}$.
In this paper we  write $y_i^{p-1}$ by $Y_i$, and $v^{p-1}$ by $V$.

  From the  Poincare series and formula (1), we get the another
  expression of $H^*(E)$ (Proposition 9 in [Gr-Le], or [Ya2])
  \[(2)\quad H^*(E)\cong k[C,v]\{y_1^iy_2^j|0\le i,j\le p-1,
  (i,j)\not =(p-1,p-1)\}.\]

Let us  write $(\bZ/p)^2$  by $A$ simply.
The $E$ conjugacy classes of $A$-subgroups are written by
\[A_i={\langle}c, ab^i{\rangle}\ for\ 0\le i\le p-1,\quad A_{\infty}={\langle}c,b{\rangle}.\]
For $A=A_i$ some $i$, if we take $y,u$ as $H^*(A)\cong k[y,u]$,
then $C|_A=Y$ and $v|_A=u^p-y^{p-1}u$.
 The transfer map is given by $Tr_{A_0}^E(y)=0$ and
 \[ (3)\quad Tr_{A_j}^E(u^{i})=\begin{cases}
 (jy_1+y_2)^{p-1}-C\quad 
if \ i=p-1\\
                             0\quad for \ i<p-1
                                 \end{cases} \]
                                (for $j =\infty$, we have $tr_{A_{\infty}}^E(u^{p-1})=y_1^{p-1}-C$).

We first  consider the $Out(E)$-module  decomposition of $H^*(E)$. 
Recall that $Out(E)\cong Out(A)\cong GL_2(\bF_p)$. The simple
modules of $G=GL_2(\bF_p)$ are well known.  Let us think of 
$A$ as  the natural two-dimensional representation, and $det$ the determinant representation of $G$. 
Then there are $p(p-1)$ simple $k[G]$-modules given by $S(A)^i\otimes (det)^q$ for $0\le i\le p-1,0\le q\le p-2$ where
\[S(A)^i=k\{y_1^i,\ y_1^{i-1}y_2,\ ,...,\ y_2^i\}\cong H^{2i}(E).\]  
(This $S(A)^i$ is isomorphic to $S(A)^i$ in $\S 4$, 
but take generators $y_1$ (resp. $y_2$) for  $y$ (resp. $u$).

    Recall that $k\{v\}\cong det$ as $Out(E)$-modules.
       Note that  
        \[ \bC\bA=k[C,V]\cong H^*(E)^{Out(E)} . \]

 For   
$j=(p-1)+i$ with $0\le i\le p-2$.  Write it
\[H^i(E)\supset T(A)^i,\quad
T(A)^i=k\{y_1^{p-1}y_2^i,\ y_1^{p-2}y_2^{i+1},\ 
...,\ y_1^iy_2^{p-1}\}.\] 
(Note $S(A)^{p-1}=T(A)^0$.)
Using the the relation $d_2=y_1^py_2-y_1y_2^p=0$ in $H^*(E)$,
 we can consider $T(A)^i$ is an $Out(E)$-module
such that $T(A)^i\cong S(A)^{p-1-i}\otimes det^i[2i]$
from Lemma 4.1.
In fact, from (2), we also have
\[H^*(E)\cong k[C,v]\otimes (\oplus_{i=0}^{p-2}(S(A)^i\oplus T(A)^i)).\]
Hence we have 
              \begin{thm} (Theorem 4.4 in [Hi-Ya1])
              There is a decomposition of $Out(E)$-module such that
              \[ H^*(E)\leftrightarrow  \bC\bA\otimes
              (\oplus_{q=0}^{p-2}\oplus _{i=0}^{p-2}(
     S(A)^i\otimes v^q\oplus T(A)^i\otimes v^q))\]
         \[where \quad  S(A)^i\otimes v^q\cong S(A)^i\otimes det^q ,\quad
        T(A)^i\otimes v^q\cong S(A)^{p-1-i}\otimes det^{i+q}[2i].\]
                \end{thm}
         Let us write by $H_{i,q}$ the  summand of $H^*(E)$ which is a sum
         of the all (sub and quotient) modules isomorphic to $S(A)^i\otimes det^q$.
(In the notation in $\S 3$,
$H_{i,j}^*\cong (i+1)H^*(Y_{i,j}).)$
         \begin{cor}  We have the $Out(E)$-module
decomposition 
\[ H_{i,q}\leftrightarrow  \begin{cases} \bC\bA\otimes v^q\qquad for\ i=0,\\
            \bC\bA\otimes T(A)^0\otimes v^q\quad (T(A)^0=S(A)^{p-1})\quad for \ i=p-1,\\
           \bC\bA\otimes(S(A)^i\otimes v^q\oplus
            T(A)^{p-1-i}\otimes v^{i+q})\quad otherwise.
\end{cases}\]
           \end{cor}

(I) $P=C(r)$, $r>3$.

By Dietz and Priddy, the  stable splitting is known.
The splitting is given as
\[BP\cong \vee (i+1)X_{i,q}\vee (q+1)L(1,q)\vee pL(1,p-1)\]
where $0\le i\le p-1$, $0\le q\le p-2$
and $L(1,p-1)=L(1,0)$. 
Transfers from  proper subgroups are  always zero
when $r>3$.
We have
\[Y_{i,q}=\begin{cases} X_{i,q}\quad q\not =0\\
                                X_{i,0}\vee L(1,i)\quad q=0.
\end{cases} \]

\begin{thm}  Let $P=C(r)$ and $r\ge 4$.  Then 
\[(i+1)H^*(X_{i,q})\cong \begin{cases}
                      H_{i,q}\quad if \ q\not =0\\
                     \bC\bA\otimes(S^i(A)\}\{V\}\oplus T^{p-1-i}(A)v^{i})\quad q=0, i\not=p-1\\
\bC\bA\otimes S^{p-1}(A)\{V\}\quad q=0,\ i=p-1.
\end{cases}\]
\end{thm}
\begin{proof}
We only need to prove $q=0$.
Note that  $\bC\bA\cong k[C,V]\cong k[C]\oplus \bC\bA\{V\}$.
So we have  $\bC\bA\ominus k[C]\cong \bC\bA\{V\}$.
Then we can compute as
\[ \bC\bA\otimes S^i(A)\ominus (i+1)k[C]\{y^i\}\]
\[ \cong \bC\bA\{V\}
\{y_1^i,y_1^{i-1}y_2,...,y_2^i\}
\cong \bC\bA\otimes S(A)^i\{V\}.\]
Using this and 
\[(i+1)H^*(X_{i,0})\cong (i+1)H^*(Y_{i,0})\ominus (i+1)H^*(L(1,i))\]
we can get the theorem (for $q=0$).
\end{proof}

(II) $C(3)=p_+^{1+2}$.
  
In this case, the decomposition of 
cohomology is given in [Hi-Ya1] but it is quite complicated.  
By Dietz-Priddy,
 the splitting is given as
\[BP\cong \vee (i+1)_{i,q}X_{i,q}\vee \vee_k (p+1)L(2,q)\vee_q (q+1)L(1,q)\vee pL(1,p-1)\]
where $0\le i\le p-1$ and $0\le q\le p-2$.
The different places from $r\ge 4$ are  the existence of $L(2,q)$ which are induced from the transfer
(see $\S 9$ in [Hi-Ya1] for details).
\begin{lemma}
We have the isomorphisms 
\[ (p+1)\oplus_{q=1}^{p-1}H^*(L(2,q))\cong
   \bC\bB\otimes (\oplus _{q=1}^{p-1}\oplus_{j\in \bF_p\cup \infty}(1)_{q,j})\qquad\] $\ \ \ \ \ $
$\qquad$ $\qquad$ $\qquad$ $\qquad$  $\cong \bC\bB\otimes (\oplus _{q=1}^{p-1}(2)_q)$ where 
\[\qquad  (1)_{q,j}=k\{Tr_{A_j}^E(u^{p-1})d_2(A_j)^q\}\cong k\{ ((jy_1+y_2)^{p-1}-C)d_2(A_j)^q\},\] 
\[ (2)_q=
\begin{cases} (S(A)^q\{C\otimes v^q\}\oplus
T(A)^q\otimes v^q)\quad 1\le q\le p-2\\
        (S(A)^{p-1}\oplus k\{C\})\{D_2\}\quad q=p-1
\end{cases}\]
where $d_2(A_j)=v(y_1+jy_2)$.
\end{lemma}
\begin{proof}[Outline of Proof]
(See $\S 9$ in [Hi-Ya1] for details.)
Let $x\in (1)_{q,j}\subset 
(k\{C\}\oplus S(A)^{p-1})\otimes d_2^q$.
Then  for $\mu, \lambda_i \in \bZ/p$, the element $x$ is written as 
\[ \mu C\otimes d_2^q+\sum_{i}\lambda_i y^i_1y_2^{p-1-i}\otimes d_2^q\]
\[=
\mu Cy_1^q\otimes v^q+\sum_i\lambda_i y^{i+q}_1y_2^{p-1-i}\otimes v^q\]
\[ \in (S(A)^q\{C\otimes v^q\}\oplus T(A)^q\otimes v^q)
\subset H_{q,q}\oplus H_{p-1-q,2q}.\]
The last inclusion follows from 
$T^q\otimes v^q\cong S(A)^{p-1-q}\otimes det^{q+q}$
as $Out(P)$-modules.
Hence we see 
$\oplus_{j\in \bF_p\cup \infty}(1)_{q,j} \subset (2)_{q}$.
Since 
\[(1)_{q,j}|A_k=\begin{cases} (y_1+jy_2)^{p-1}d_2(A_j)^q\quad j=k\\
            0\qquad otherwise,
            \end{cases}\]
            we have $dim(\oplus_{j\in \bF_p\cup \infty}(1)_q)\ge p+1$.
            Of course (from Theorem 6.1), we have $dim(2)_q= p+1$.
            Therefore we show that $\oplus_{j\in \bF_p\cup \infty}(1)_q=(2)_q$.
                   \end{proof} 
Note that
$(p+1)H^*(L(2,q))\subset H_{q,q}\oplus H_{p-1-q,i+q},$
however each $H^*(L(2,q))$ is contained in either
$H^*(Y_{q,q})$
or $ H^*(Y_{p-1-i,i+q})$.
\begin{thm}
The $Out(E)$-module decomposition
\[ H^*(E)\leftrightarrow \oplus_{i,q} H_{i,q}\cong
\oplus_{i,q} \bC\bA\otimes (S(A)^i\otimes v^q\oplus T(A)^{p-1-i}\otimes
v^{2q})\]
gives  simple $A(E,E)$-modules decomposition
by  each of the following sums of $Out(E)$-simple modules  
\[(1)\quad S(A)^{p-1}\oplus k\{C\} \]
\[(2)\quad S(A)^q\{C\}\otimes v^q
\oplus T(A)^q\otimes v^q \]
as one $A(E,E)$-simple module
 for  $1\le q\le p-2$. That is
 \[(1)\cong S(E,\bZ/p,det^0)[2p-2],\quad
 (2)\cong S(E,A, S(A)^{p-1}\otimes det^q)[2(p-1)q].\]
\end{thm}
\begin{proof}[Outline of Proof.]
We prove that $(1)$ is a simple $A(E,E)$-module. 
We consider $\Phi_1,\Phi_2\in A(E,E)$
\[\Phi_1:P> A_{0}\to \la c\ra \stackrel{w}{\cong} 
\la a \ra \subset A_{0}\subset P,\] 
\[\Phi_2:P> A_{\infty}\to \la c\ra \stackrel{w}{\cong} 
\la b \ra \subset A_{\infty}\subset P.\]
Then we see 
$(\Phi_1-\Phi_2)(C)=y_1^{p-1}-y_2^{p-1}\in S^{p-1}(A),$ and 
\[\Phi_1(y_1^{p-1})=y_2^{p-1}-C\in S^{p-1}(A)\oplus k\{C\}.\]
Let $B=(S^{p-1}(A)\oplus k\{C\})$.  Then from the first equation, 
$B/k\{C\}$ is not an $A(P,P)$-module.
From the second one, we see that $B/S^{p-1}(A)$ is also not an
$A(P,P)$-module.  Hence $B$ is a simple $A(P,P)$-module.

The fact that $(2)$ is isomorphic to a  $A(P,P)$ simple module is proved similarly 
using  $\Phi_1'$ and $\Phi_2'$ defined by
\[\Phi_1':P> A_0
\stackrel{a\leftrightarrow c}{\cong}  A_0\subset P,
\quad 
\Phi_2':P> A_{\infty}\stackrel{b\leftrightarrow c}{\cong}
 A_{\infty}\subset E.\]
\end{proof}
Note that
$dim(2)=(q+1)+(p-1-q+1)=p+1.$
In fact, this is the number of $L(2,q)$ in $BE$.

Let $i=0$ or $p-1$.  Then we see
\[ Y_{i,q}\cong \begin{cases} X_{i,q}\quad q\not =0\\
                      X_{i,0}\vee L(2,0)\vee L(1,0)\quad q=0.
\end{cases}\]
Using this we can prove
\begin{thm} (Corollary 10.7 in [Hi-Ya1])
We have for $1\le q\le p-2$,
\[H^*(X_{0,0})\cong \bD\bA, \quad pH^*(X_{p-1,0})\cong \bD\bA\otimes S(A)^{p-1}\{V\},\]
\[H^*(X_{0,q})\cong H_{0,q}\cong \bC\bA\{v^q\}, \quad pH^*(X_{p-1,q})\cong H_{p-1,q}
\cong \bC\bA\{S(A)^{p-1}\otimes v^q\}.\]
 \end{thm}
 \begin{proof}[Proof of the first isomorphism]
Recall $H^*(L(2,0)\vee L(1,0))\cong \bC\bB\{C\}$
from Lemma 5.4.  Hence  we see
\[H^*(X_{0,0})\cong \bC\bA\ominus \bC\bB\{C\}\cong \bD\bA.\]
The other cases are proved similarly.
\end{proof}
 
 {\bf Examples.}
See $\S 6$ in [Ya2] for $p=3$ case.
For the sporadic  simple group $J_4$ and the twisted Chevalley group $^2F_4'$,
we have the isomorphisms 
\[ H^*(J_4)\cong H^*(X_{0,0})\cong  \bD\bA,\quad (Green\ [Gr])\]
\[H^*(^2F_4')\cong H^*(X_{0,0}\vee X_{2,0})\cong \bD\bA\{1,YV\},\]
\[H^*(E)^{SD_{16}}\cong 
H^*(X_{0,0}\vee L(2,0)\vee L(1,0))\cong \bC\bA,\]
%\[ H^*(E)^{Q_8}\cong H^*(X_{0,0}\vee M(2)\vee X_{0,1})
%\cong  \bC\bA \{ 1,v \},  \]
%Here $M(2)=L(2,0)\vee L(1,0)$. and %$H^*(M(2))\cong\bC\bB\{C\}$.

 \begin{thm} (Corollary 10.8 in [Hi-Ya1])
Let $1\le i\le p-2$. 
Let us write simply
\[ S=S(A)^i\otimes v^q,\quad T=T(A)^{p-1-i}\otimes v^{i+q}.\]
Then we have the isomorphism
\[(1+i)H^*(X_{i,q})\cong
\begin{cases}
\bD\bA\otimes(S\oplus T\{V\})\quad if\ i=q\not =0,\ 3q\equiv 0\ (mod(p-1))\\
          \bD\bA\otimes S \oplus \bC\bA\otimes T
\quad if\ i=q,\ 3q\not \equiv 0\\
\bC\bA\otimes S\oplus \bD\bA\otimes T\{V\}\quad
   if \ q\equiv-2i\not\equiv0,\ 3i\not \equiv 0\\
\bC\bA\otimes S\{V\}\oplus \bD\bA\otimes T\{V\}
\quad if \  q\equiv0,\ 2i\equiv0\\
\bC\bA\otimes S\{V\}\oplus \bC\bA\otimes T
\quad if \  q\equiv0,\ 2i\not \equiv0\\
H_{i,q}\cong \bC\bA\otimes (S\oplus T)\quad   otherwise.
\end{cases} \]
\end{thm}
\begin{proof}[Outline of Proof.]
From the proof of Lemma 5.5, we see
\[ \oplus _jTr_{A_j}^E(H^*(L(2,q))\subset  H_{q,q}\oplus H_{p-1-q,2q}.\]
We prove the first isomorphism.
Suppose $3q\equiv 0,\ q\not \equiv 0\ mod(p-1)$.  Then
( see the proof of Lemma 6.4)
\[\oplus _jTr_{A_j}^E(H^*(L(2,2q))\subset  H_{2q,2q}\oplus H_{q,q}\]
since $H_{p-1-2q,4q}=H_{q,q}$.  Using this we can prove
\[ Y_{q,q}\cong X_{q,q}\vee L(2,q) \vee L(2,2q).\]
Hence we can compute
\[(q+1)H^*(X_{q,q})\cong (\bC\bA\otimes S\ominus (q+1)\bC\bB\{Cd^{q}\})\]
 \[ \oplus  (\bC\bA\otimes T\ominus (q+1) \bC\bB\{Cd^{2q}\}).\]
Here $d^q=y^qv^q\in S=S^q(A)\{v^q\}$ and 
$Cd_2^{2q}=(Cy^{2q})v^{2q}\in T=T(A)^{2q}\{v^{2q}\}$.
Hence 
\[\bC\bA\otimes S\ominus (q+1)\bC\bB\{Cd^{q}\} \cong 
      (\bC\bA\ominus \bC\bB\{C\})\{S^q(A)v^q\}\cong \bD\bA\otimes S.\]
\[\bC\bA\otimes T\ominus (q+1)\bC\bB\{Cd^{2q}\} \cong 
      (\bC\bA\ominus \bC\bB)\{T^{2q}v^{2q}\}\cong \bD\bA\{C^p\}\otimes T.\]
Note $D_1=C^p+V$.
Thus we can prove the first isomorphism.  The other isomorphisms are proved similarly.
\end{proof}
We write down here the splitting in the all cases.
\begin{cor}  For $1\le i\le p-2$,  
there are stable homotopy equivalences
 \[Y_{i,q}\cong
\begin{cases}
X_{q,q}\vee L(2,q)\vee L(2,2q)\\
\qquad \quad  if\ i=q\not =0,\ 3q\equiv 0\ (mod(p-1))\\
X_{q,q}\vee L(2,q)\quad if\ i=q,\ 3q\not \equiv 0\\
X_{i,-2i}\vee L(2,-i) \quad
   if \ q\equiv -2i
\not\equiv 0,
\ 3i\not 
\equiv 0\\
X_{i,0}\vee L(1,i)\vee  L(2,-i) 
\quad if \  q\equiv0,\ 2i\equiv0\\
X_{i,0}\vee L(1,i) \quad if \  q\equiv0,\ 2i\not \equiv0\\
X_{i,q} \quad   otherwise.
\end{cases} \]
\end{cor}

{\bf Example.}  
When $p=7$ (see $\S 9$ in [Ya2]),  we see the  cohomology   
\[ H^*(X_{0,0})\cong \bD\bA,\ H^*(X_{6,0})\cong \bD\bA\{a^3\},\
H^*(X_{4,4})\cong \bD\bA\{a^2,a^4\},\]
\[H^*(X_{2,2})\cong \bD\bA\{a,a^5\}.\]
where $a=s^2\otimes v^2\in S(2,2)$, $a^5=s^{10}\otimes v^{10}= s^{10}\otimes v^4 V \in T(2,2)$,...
and where $S(i,q)=S, T(i,q)=T$ in the preceding theorem.
Here $a^6=D_2^2$ (page 416 in [Ya2]).
Thus we see the cohomology of the exotic finite $7$-groups
(see $\S 9$ in [Ya2]) found by Ruiz and Viruel [Ru-Vi]
\[H^*(RV_3)\cong H^*(X_{0,0}\vee X_{4,4})
\cong \bD\bA\{1,a^2,a^4\},\]
\[H^*(RV_2)\cong H^*(X_{0,0}\vee X_{4,4}\vee X_{6,0}\vee X_{2,2})
\cong \bD\bA\{1,a,a^2,a^3,a^4,a^5\}\]
\[H^*(RV_1)\cong H^*(X_{0,0}\vee X_{6,0}\vee X_{4,4})
\cong \bD\bA\{1,a^2,a^3,a^4\}.\]
Therefore there does not exist even a $7$-local finite group $G$ such that
$H^*(G)\cong \bD\bA$.

\section{ $G(r,e)$ for $p\ge 5$}

For $p\ge 5$,  groups $P$ with $rank_pP=2$ are classified by Blackburn (see Thomas [Th], Dietz-Priddy [Di-Pr],
[Ya1] ).
They are metacyclic groups,  groups $C(r)$ and $G(r',e)$.
Throughout this section, we assume $p\ge 5$.

  The group $G=G(r,e), r\ge 4$ (and $e$ is $1$ or a quadratic non residue modulo $p$) is defined as
\[\la a,b,c|a^p=b^p=c^{p^{r-2}}=[b,c]=1, [a,b^{-1}]=c^{ep^{r-3}},[a,c]=b\ra.\]
The subgroup $\la a,b,c^{p}\ra$ is isomorphic to $C(r-1)$.
Hence we have the extension
\[ 1\to C(r-1)\to  G(r,e) \to \bZ/p \to 0.\]

Of course $E=C(3)\subset C(r-1)\subset G(r,e)$.
By [Ya1], we have an isomorphism
\[ H^*(G(r,e))\cong H^*(E)^{\la c\ra}.\]
Indeed, in Theorem 5.29 in [Ya1], we see $H^{even}(G;\bZ)
\cong (Y_1\oplus Y_w\oplus C')\otimes C_p'$ and
in (5.5), we see $ H^{even}(E;\bZ)^{\la c\ra}\cong (Y_1\oplus Y_w\oplus C)\otimes C_p.$
Here we can show 
\[C_p/p\cong C_p'/p,\quad C/p\cong \begin{cases} C'/p\quad for\ r=4\\
                        C'/(p,c_1) ,\quad ( c_1: nilpotent)\  for\ r\ge 5.
\end{cases}\]
The invariant ring $H^*(C(3))^{\la c\ra}$  is multiplicatively generated by 
\[y_1,\ C,\ v,\  y_2^iw\quad where\ w=y_2^p-y_1^{p-1}y_2,\quad 0\le i\le p-3\]
since $c^*: y_2\mapsto y_2+y_1$ and   $C^2=Y_1^{2}+y_2^{p-2}w$.      Hence we have
\begin{lemma}  We have an isomorphism
\[(1)\quad H^*(G(r,e))\cong (k[y_1]\oplus k[y_2]\{w\}\oplus k\{C\})\otimes k[v]\]
where the multiplications are given by $y_1w=0$, $Cy_1=y_1^p$, $w^2=y_2^pw$ \\
and $Cw=y_2^{p-1}w$,
Thus we also have the isomorphism
\[(2)\ \ H^*(G(r,e))\cong \bC\bA(\oplus_{q=0}^{p-2}(k\{1,y_1,...,y_1^{p-1}\}\{v^q\}\oplus k\{1,y_2,...,y_2^{p-3}\}\{wv^q\}).\]
\end{lemma}
Here we note that
\[ H^*(E)^{\la c\ra}\cap \oplus _{i=0}^{p-2}S(A)^i\cong k\{1,y_1,...,y_1^{p-2}\},\]
\[ H^*(E)^{\la c\ra}\cap \oplus_{i=0}^{p-2}T(A)^i\cong k\{y_1^{p-1}\}\oplus k\{1,y_2,...,y_2^{p-3}\}\{w\}.\]
Let us write $w_{i+1}=y_2^iw$ (so $w_1=w$) and 
\[ S(G)=k\{1,y_1,...,y_1^{p-2}\},\quad
T(G)=k\{y_1^{p-1},w_1,..., w_{p-2}\}\]
so that 
$H^*(G(r,e))\cong \bC\bA\otimes ( \oplus _{i}(S(G)\oplus T(G))\{v^i\}).$

For groups $G'=G(r,e), E,...$, let us write by $Y_{i,j}(G')$ (and $X_{i,j}(G')$) the decomposition
component for $BG'$.  Then  from Corollary 6.2, we have
\begin{lemma}  We have additively
\[ H^*(G(r,e))\cong  \oplus _{i,q}H^*(Y_{i,q}(E))
\qquad with \]
\[ H^*(Y_{i,q}(E))\cong   \begin{cases} \bC\bA\otimes v^q\qquad for\ i=0,\\
            \bC\bA\otimes k\{y_1^{p-1}\otimes v^q\}
\quad for \ i=p-1,\\
           \bC\bA\otimes(k\{y_1^i\otimes v^q,
w_{p-1-i}\otimes v^{i+q}\})\quad otherwise
\end{cases} \]
where $0\le i\le p-1$ and $0\le q\le p-2$.
\end{lemma}

The outer automorphism  is $Out(P)\cong (p-group):(\bZ/2\times \bZ/(p-1))$ (see [Di-Pr] for details).  Here the action
$i\in \bZ/2$ induces $i:a\mapsto a^{-1}$ and $k\in \bZ/(p-1)$ induces $k:c\mapsto c^{k}$.  Hence
\[ i^*:\begin{cases}y_1\mapsto -y_1\\
                      y_2\mapsto -y_2,
\end{cases}\quad k^*:
\begin{cases}v\mapsto kv\\
y_2\mapsto ky_2.
\end{cases}\]
All simple $\bZ/2\times \bZ/(p-1)$-modules  are represented as
\[ k\{v^i\},\quad k\{y_1v^i\}\quad 0\le i\le p-2.\]
Using this and Lemma 7.2 (2),  we get 
\begin{lemma}  Let $P=G(r,e)$ with $r\ge 4$.  Then we have $Out(P)$-module decomposition                          
\[H_{i,q}\leftrightarrow  H^*(Y_{i,q}(P))\cong \begin{cases}
\oplus _{j=even}H^*(Y_{j,q}(E)) \quad if \ i=0\\
\oplus _{j=odd}H^*(Y_{j,q}(E)) \quad if \ i=1\end{cases}\]
where $0\le i\le1$, $0\le j\le p-1$ and $0\le q\le p-2$.
\end{lemma}

(I) The  case  $P=G(r,e)$ and  $r>4$.

The stable splitting is given by Dietz-Priddy [Di-Pr]
\[ BG(r,e)\cong \vee_{i,q}X_{i,q}(G(r,e))\vee \vee_qX_{p-1,q}(C(r-1))\vee \vee_qL(1,q)\]
where $i\in \bZ/2$ and $0\le q\le p-2$.

From Theorem 6.3,  (for $r-1\ge 4$)  recall
\[   pH^*(X_{p-1,q}(C(r-1))\cong \begin{cases} \bC
\bA\otimes S(A)^{p-1}\{V\}\quad if\ q=0\\
                                                \bC\bA\otimes S(A)^{p-1}v^q \quad if \ 0<q<p-1.
\end{cases} \]
This summand induced from the following transfer.
Recall $[a,c]=b$ in $G(r,e)$ and $c^*(y_2)=y_2+y_1$ in $H^*(\la a,b,c^p\ra)=H^*(C(r-1))$.
Hence we can compute
\[Tr_{C(r-1)}^G(y_2^{p-1})|_{C(r-1)}=\sum_{j}(y_2-jy_1)^{p-1}=-y_1^{p-1}
\]
which implies that $Tr_{C(r-1)}^G(y_2^{p-1})=-y_1^{p-1}$
since $H^*(P)\subset H^*(C(r-1))$.
Define $\Phi\in A(P,P)$ by 
\[ \Phi:P>C(r-1)\stackrel{a\leftrightarrow b}{\to}
      C(r-1)\subset P.\]
Then we have
\[\Phi(v^qy_1^{p-1})=T
r_{C(r-1)}^P(v^qy_2^{p-1})=-v^qy_1^{p-1}.\]
Hence $\Phi(H^*(Y_{p-1,q}(E)))\cong
\bC\bA\{y_1^{p-1}v^q\}\cong H^*(Y_{p-1,q}(E)).$
Thus we have
\[Y_{i,q}\cong \begin{cases}
          X_{0,0} \vee (X_{p-1,0}(C(r-1))\vee \vee_{j=ev}^{p-3} L(1,j) \quad if\ i=q=0\\
          X_{0,q}  \vee X_{p-1,q}(C(r-1)) \quad if\ i=0,\ q\not =0\\
           X_{1,0}\vee \vee_{j=odd}L(1,j)\quad if\ i=1,q=0\\
             X_{1,q}\quad if\ i=1,\ q\not =0.\\
\end{cases}\]
\begin{thm}
For $r>4$, we have
\[ H^*(X_{i,q}(G(r,e))\cong \begin{cases}
         H^*((\vee_{j=ev}^{p-3}X_{j,0}(C(r-1)))\vee L(1,0))\quad
if \ i=q=0\\
 H^*(\vee_{j=ev}^{p-3}X_{j,q}(C(r-1)))\quad
if \ i=0, q\not =0\\
H^*(\vee_{j=odd}
^{p-2}X_{j,q}(C(r-1))\quad if\ i=1
\end{cases}\]
\[ \cong \begin{cases}
\bC\bA\otimes(k\{1\} \oplus \oplus_{j=ev>0}^{p-3}k\{y_1^jV,w_{-j}v^j\})\quad if\ 
i=q=0\\
\bC\bA\otimes(\oplus_{j=ev}^{p-3}k\{y_1^jv^q,w_{-j}jv^{j+q}\})\quad if\ 
i=0,q\not =0\\
\bC\bA\otimes(\oplus_{j=odd}^{p-2}k\{y_1^jV,w_{-j}v^j\})\quad if\ 
i=1,q=0\\
\bC\bA\otimes(\oplus_{j=odd}^{p-2}k\{y_1^jv^q,w_{-j}v^{j+q}\})\quad if\ i=1, q\not=0.\end{cases}\]
where $w_{-j}=w_{p-1-j}$.
\end{thm}
\begin{proof}
We will prove the case $i=q=0$, and other cases are
proved similarly.
From the above isomorphism for $Y_{i,q}$, we have
\[H^*(X_{0,0})\cong H^*(Y_{0,0})\ominus H^*(X_{p-1,0}(C(r-1)))\ominus \oplus_{j=ev}^{p-3}H^*(L(1,j)).\]
Using Lemma 7.3, we see
$H^*(Y_{0,0})\cong \oplus_{j:ev}^{p-1}H^*(Y_{j,0}(C(r-1)).$
Hence  we can write $ H^*(X_{0,0})\cong A\oplus B$
with 
\[A=\oplus _{j=ev}^{p-3}H^*(Y_{j,0}(C(r-1))\ominus 
\oplus_{j=ev}^{p-3}H^*(L(1,j)),\]
\[ B=H^*(Y_{p-1,0}(C(r-1))
\ominus H^*(X_{p-1,0}(C(r-1)).\]
Here we have
\[A \cong \oplus _{j=ev}^{p-3}H^*(X_{j,0})(C(r-1))\quad and\quad B\cong H^*(L(1,0)).\]
Thus we have the first isomorphism in the theorem for $i=q=0$.

The second isomorphism follows from
\[H^*(X_{0,0})\cong A\oplus H^*(L(1,0))\cong 
 \bC\bA\{1\} \oplus \oplus _{0<j=ev}^{p-3}H^*(X_{j,0}(C(r-1)).\]
Here we used
$ H^*(X_{0,0}(C(r-1)))\oplus H^*(L(1,0))\cong \bC\bA\{1\}.$
\end{proof}

%For ease of notations, let us write $w_{-i}=w_{p-1-i}$.

% Then
%\[H^*(X_{i,k})\cong \begin{cases}
%\bC\bA\otimes ( \oplus_{j=ev\ge 0}
%k\{y_1^jV\}\oplus \oplus_{j=ev>0}\{ w_{-j}v^{j+k}\}
%) \\
%\quad \oplus \bC\bB\{C^2V\}
%\quad if\ i=k=0
%\\
%\bC\bA\otimes ( \oplus_{k\not =j=ev\ge 0}
%k\{y_1^jv^k\}\oplus
% \oplus_{j=ev>0}k\{w_{-j}v^{j+k}\}) \\
%\qquad \oplus k[V]\{y_1^kv^k\} 
% \quad if\  i=0,k=ev\not =0
%\\
%\bC\bA\otimes ( \oplus_{j=ev\ge 0}
%k\{y_1^jv^k\}\oplus
% \oplus_{j=ev>0}k\{w_{-j}v^{j+k}\}) \\
%\quad  \qquad  if\  i=0,k=odd\not =0
%\\
%\bC\bA\otimes ( \oplus_{j=odd}
%k\{y_1^jV\}\oplus
% \oplus_{j=odd>0}k\{w_{-j}v^{j+k}\}) \\
%\quad  \quad if\  i=1,\ k=0\\
%
%
%\bC\bA\otimes ( \oplus_{j=odd}
%k\{y_1^jv^k\}\oplus
% \oplus_{-k/2\not =j=ev>0}k\{w_{-j}v^{j+k}\}) \\
% \qquad  \quad if\  i=1,\ k=even>0\\
%
%
% \bC\bA\otimes ( \oplus_{k\not=j=odd}
%k\{y_1^jv^k\}\oplus \oplus_{j=odd} w_{-j}v^{j+k}\}
%)\\
%\qquad \oplus k[V]\{y_1^kv_k\}
%\quad if\ i=1,k=odd.
%\end{cases} \]
%\end{thm}
%
%\begin{proof}
%We use the facts that
%$\bC\bA\ominus\bC\bA\{C\}\cong k[V]$ and
%$\bC\bA\ominus\bD\bA\cong \bC\bB\{C\}$.
%\end{proof}
 
(II) $r=4$.

In this case cohomology is the same as (I).
However the stable splitting is not same as $(I)$ and it is also given by Dietz and Priddy {Di-Pr]. 
\[ BG(r,e)\cong \vee_{i,q}X_{i,q}(G(r,e))\vee \vee_qX_{p-1,q}(C(r-1))\]
\[\qquad \vee_qL(2,q)\vee \vee_qL(1,q)
\]
where $i\in \bZ/2$ and $0\le q\le p-2$.

The problems are only to see that these
$H^*(L(2,q))$ go to  what $H^*(Y_{i,q'})$.
Let us consider $\Phi\in A(P,P)$ such that
\[\Phi : P>E>\la a,c^p\ra \stackrel{a\leftrightarrow c^p}{\to}
\la a,c^p\ra \subset P.\]
Then we can compute (for $d_2=y_1v$)
\[ \Phi(d_2^q y_1^{p-1})=Tr_E^PTr_{\la a,c^p\ra}^E(d_2^qu^{p-1})
=Tr_E^P(d_2^q(y_2^{p-1}-C))\]
\[=Tr_E^P(d_2^q(y_2^{p-1}-C))|_{E}=-d_2^qy_1^{p-1}\]
from (3) in $\S 6$ and the arguments before Lemma 7.4. 
This means $y_1^{q+p-1}v^k$ is in the image from
$H^*(L(2,q))\subset H^*(\la a,c^p\ra) $.
Note if $q$ is even, then $y_1^{k+p-1}v^q\in H^*(Y_{0,q})$,
otherwise it is in $H^*(Y_{1,q})$.
Hence we see 
\[ Tr_{\la a,c^p\ra}^PH^*(L(2,q))\subset  \begin{cases}  H^*(Y_{0,q})\quad q=even\\
                           H^*(Y_{1,q})\quad q=odd.
\end{cases}\]

In particular, note that 
\[D_2y_1^{p-1}\in H^*(Y_{p-1,0}(E))\subset H^*(Y_{0,0}(P))\]
 is in the image from $H^*(L(2,0))$.
However note  $D_2C\in H^*(Y_{0,0}(E))\subset H^*(Y_{0,0}(P))$ is not in the image from $H^*(L(2,0))$.
 while it is so in $H^*(Y_{0,0}(E))$. 

Thus we have
\[Y_{i,q}\cong \begin{cases}
          X_{0,0} \vee (X_{p-1,0}(E)\vee L(2,0)\vee L(1,0))\vee \vee_{j=ev>0}^{p-3} L(1,j)\\
 \quad \qquad \qquad if\ i=q=0\\
          X_{0,q}  \vee X_{p-1,q}(E) \vee L(2,q)\quad if\ i=0,\ q=ev\not =0\\
  X_{0,q}  \vee X_{p-1,q}(E) \quad  if\ i=0,\ q=odd\not =0\\           X_{1,0}\vee \vee_{j=odd}L(1,j)\quad if\ i=1,q=0\\
             X_{1,q}\quad if\ i=1,\ q=ev\not =0.\\
           X_{1,q}\vee L(2,q)\quad if\ i=1,\ q=odd\not =0.\\
\end{cases}\]
Then we have 
\begin{thm}
For $P=G(4,e)$, the cohomology $H^*(X_{i,q})$ is isomorphic to
\[ \cong \begin{cases}
\bC\bA\otimes(k\{1\} \oplus \oplus_{0<j=ev}^{p-3}k\{y_1^jV,w_{-j}v^j\})\quad if\ 
i=q=0\\
\bC\bA\otimes(\oplus_{0<j=ev\not =q}^{p-3}k\{y_1^jv^q,w_{-j}v^{j+q}\})\\
\qquad\oplus \bD\bA\{y_1^qv^q\}\oplus \bC\bA\{w_{-q}v^{2q}\} 
\quad if\ i=0,q=even \not =0\\
\bC\bA\otimes(\oplus_{j=ev}^{p-3}k\{y_1^jv^q,w_{-j}jv^{j+q}\})\quad if\ 
i=0,q=odd \not =0\\
\bC\bA\otimes(\oplus_{j=odd}^{p-2}k\{y_1^jV,w_{-j}v^j\})\quad if\ 
i=1,q=0\\
\bC\bA\otimes(\oplus_{j=odd}^{p-2}k\{y_1^jv^q,w_{-j}v^{j+q}\})\quad if\ i=1, q=even\not=0\\
\bC\bA\otimes(\oplus_{j=odd\not =q}^{p-3}k\{y_1^jv^q,w_{-j}v^{j+q}\})\\
\qquad\oplus \bD\bA\{y_1^qv^q\}\oplus \bC\bA\{w_{-q}v^{2q}\} 
\quad if\ i=1,q=odd \not =0.
\end{cases}\]
\end{thm}
\begin{proof}
When $i=q=0$,  the isomorphism follows from
\[H^*(X_{0,0}(G(4,e))\cong H^*(X_{0,0}(G(r,e)))\quad for 
\ r>4.\]
This fact is shown from the decomposition $Y_{0,0}$ above
and
\[ H^*(X_{p-1,0}(C(r-1))\cong H^*(X_{p-1,0}(C(3))\vee L(2,0))\quad for\  r>4.\]
When $i=0,q=ev\not=0$, the fact
\[ H^*(X_{0,q}(G(4,e))\cong H^*(X_{0,q}(G(r,e'))\ominus
H^*(L(2,q))\]
implies the isomorphism in the theorem.
The other cases are also seen similarly.
\end{proof}

\section{Relations among $BP$ with $rank_pP=2$.}

In this section, we see Theorem 1.1 in the introduction.
For a group $G$ with $rank_pG=2$, let us write by
$X_{i,q}(G)$ (or $X_i(G)$ for a metacyclic group)  the corresponding irreducible stable homotopy summand.

Recall that a non-dominant summand $X$ is the irreducible summand
corresponding to an $A(G,G)$-simple module $S(G,Q,V)$
for a proper subgroup $Q$ and a simple $Out(Q)$-module
$V$. From Dietz-Priddy, the following lemma is immediate.
\begin{lemma}  Let $G=C(r)$ (or $G(r+1,e))$ for $r\ge 3$.
Then  for $0\le q\le p-2$, a non-dominant summand is
$L(1,q),$ $L(2,q)$ (or $X_{p-1,q}(C(r))$ for $G=G(r+1,e)$).
\end{lemma} 

Let us use the notation such that for stable homotopy spaces
$A,B$, the notation
$A\cong _HB$ means $H^*(A)\cong H^*(B)$ as graded
modules.
Theorem 1.1 in the introduction is a immediate consequence
of the above lemma and the following theorem about dominant summands.
\begin{thm}  Let $G=C(r)$ (or $G(r+1,e)$) for $r\ge 3$.
Given $0\le i\le p-1$ (or $i=0\ or\ 1$) and $0\le q\le p-2$,
there are $a_{j},b_k,c$  which are $0$ or $1$ such that  we have the isomorphism
\[ X_{i,q}(G)\cong_H 
\vee_{j=0}^{p-1} a_{j}X _{j,q}(E)\vee
\vee_{k=0}^{p-2} b_{k}L(2,k)\vee cL(1,0)\quad (*)\]
  In particular,
$c=1$ if and only if  $i=q=0$ and  $G=G(r+1,e)$.
%  and 
%$b_{k,q}=2$ if and only if $P=G(r+1,e)$, $r>3$ and
%$i=0$, $k+q=3q=0\ mod(p-1)$. 
\end{thm}
\begin{proof}
We first consider $G=C(r)$, for $r>3$.  Since $C(r)\cong_H
E$, we see $Y_{i,q}(C(r))\cong_H Y_{i,q}(E)$.
Let $q=0$.  Then from the formula for $Y_{i,q}$ just before Lemma 6.3, $Y_{i,0}(C(r))\cong
X_{i,0}(C(r))\vee L(1,i)$. 

 On the other hand, from  Corollary 6.8,
we see 
\[Y_{i,0}(E)\cong_H X_{i,0}(E)\vee L(1,i)\vee bL(2,-i)\quad 
for\ b=0\ or \ 1.\]
Hence $X_{i,0}(C(r))\cong _HX_{i,0}(E)\vee bL(2,-i)$, and
 $(*)$ is satisfied in this case, in particular note
$c=0$.
The case $q\not =0$ is shown  similarly by using
Corollary 6.8
\[ Y_{i,q}(C(r))\cong _HX_{i,q}(E)\vee b'L(2,i)\vee b''L(2,-i)\quad b',b''=0\ or\ 1.\]

The case $G=G(r+1,e), r>3$ is immediate from
the first isomorphism in Theorem 7.4 and the result for $C(r)$.
We also note $c=1$ if and only if $i=q=0$. 
Hence we can write 
\[ X_{i,q}(G(r+1,e))\cong_H \vee \vee_{j=0}^{p-2}a_{j}X_{j,q}(E)\vee \vee _{k=0}^{p-2} b_{k}L(2,k)\vee cL(1,0)
.\]
where $0\le a_{j}, b_k, c \le 1$. 

The fact $0\le b_{k}\le 1$ 
is shown by the following arguments. 
Note if $Y_{j,q}(E)$ contains $b'L(2,k)$, then
$b'=1$ and $k=j$ or $k=-j$ from Corollary 6.8.
Therefore $\vee_{j=0}^{p-2} Y_{j,q}(E)$ contains $bL(2,k)$ for $b\le 2$.
Suppose $b_k=2$ and $k\not =0$. Since $Y_{j,q}(E)$ contains
$L(2,j)$ only when $j=q$, we can assume $k=q$.
However $Y_{p-1-q,q}(E)$ does not contain $L(2,q)$
from Corollary 6.8. Hence $b_k\not =2$.

 Let $G=G(4,e)$.
First note that $Y_{i,q}(G(r+1,e))\cong_H Y_{i,q}(G(4,e))$.
When $i=q=0$,  the isomorphism  $(*)$
is  immediate from
$X_{0,0}(G(4,e))\cong_HX_{0,0}(G(r+1,e))$.
When $i=0,q=ev\not=0$,  we recall 
\[ H^*(X_{0,q}(G(4,e))\cong H^*(X_{0,q}(G(r+1,e))\ominus
H^*(L(2,q))\]
Here $X_{0,q}(G(r+1,e))\cong_H \vee_{j=even}^{p-3}X_{j,q}(C(r))$.  The last space contains
\[ X_{q,q}(C(r))\cong_H X_{q,q}(E)\vee L(2,q)\vee bL(2,2q)\quad for \ b=0,\ or \ 1\]
from  Corollary 6.8, which implies the isomorphism in the theorem.
The other cases are also seen similarly.
\end{proof}

{\bf Example.}
Let $p=7$ and $r>3$.  We see  from Corollary 6.8
\[X_{5,2}(C(r))\cong Y_{5,2}(C(r))\cong_H Y_{5,2}(E)\cong 
X_{5,2}(E)\vee L(2,1).\]
We have $H^*(Y_{5,2}(E))\cong \bC\bA\{y_1^5v^2,w_1v\}$ and $H^*(L(2,1))\cong \bC\bB\{u^{p-1}d_2^1\}$ which
maps to $w_1v$ by the transfer.  Hence
\[ H^*(X_{5,2}(E))\cong H^*(Y_{5,2}(E))\ominus H^*(L(2,1)
\cong  \bC\bA\{y_1^5v^2\}\oplus \bD\bA\{w_1v\}\]
by using $\bC\bA\ominus \bC\bB\{C\}\cong \bD\bA$.
 For the case $X_{6,0}(E)$, we have
$H^*(Y_{6,0}(E))\cong H^*(X_{6,0}(E))\oplus H^*(L(2,0))\vee L(1,0))$.  Hence  we see also
\[H^*(X_{6,0}(E))\cong \bC\bA\{Y\}\ominus\bC\bB\{Y\}
\cong \bD\bA\{VY\}.\]
(See also the example after Corollary 6.8.)

{\bf Example.}  We consider the case $p=7$ and $q=2$.
The cohomology $H^*(Y_{j,2})\cong \bC\bA\{y_1^jv^2,w_{-j}v^{j+2}\}$.
Therefore for $r'>4$, we see 
\[ H^*(X_{0,2}(G(r',e)))\cong H^*(\vee_{j=0,2,4}Y_{j,2})
\cong \bC\bA\{v^2,y_1^2v^2,y_1^4v^2,w_4v^4,w_2\}\]
\[ H^*(X_{0,2}(G(4,e)))\cong H^*(\vee_{j=0,2,4}Y_{j,2})\ominus
   H^*(L(2,2))\]
\[\cong \bC\bA\{v^2,y_1^4v^2,  w_2, w_4v^4\}\oplus \bD\bA\{y_1^2v^2\}\]

Next, we study split metacyclic groups, For stable spaces
$X=X_{i,j}(G)$, 
let $SX$ 
be the virtual object 
defined by 
\[H^*(SX)= H^*(X)\cap \bC\bA\otimes (\oplus_{q=0}^{p-2}k\{1,y_1,...,y_1^{p-2}\}\{v^q\})\]
% H^*(X)/(H^*(X)\cap(\bC\bA\otimes T(A)\{v^q\})\]
where we identify  it as the submodule of
$\bC\bA\otimes(\oplus_qS(A)^*\{v^q\})\subset  H^*(E)$
in Theorem 6.1.
Then we see 
\[ H^*(S(BE))\cong \bC\bA\otimes
(\oplus_q(k\{1,y_1,...,y^{p-2}\}\{v^q\})\cong k[y,v]
\]
identifying $C=Y=y_1^{p-1}$ as graded modules.

Recall that $H^*(M(\ell,m,n))\cong k[y,v]$ with $|v|=2p^{m-\ell}$.  Let us write $M=(m-1,m,n))$ so that
\[ H^*(Y_q(M))\cong  \oplus_{j=0}^{p-2}H^*(SY_{j,q}(E)).\]
The results in $\S 5$ imply the following theorem
\begin{thm}  Let $M=M(m-1,m,n)$ and $r>3$.
 Then we have
   \[ H^*(X_q(M))\cong \begin{cases}
 \oplus_j^{p-2}H^*(SX_{j,q}(C(r))) \quad  if\  (m,n)\not=(2,1)\\
\oplus_j^{p-2}H^*(SX_{j,q}(E)) \quad if \ (m,n)=(2,1).
\end{cases}\]
\end{thm}
\begin{proof}
 The case $q\not =0$ is shown from
\[H^*(Y_q(M))\cong H^*(X_q(M))\cong H^*(X_q(M(1,2,1))\vee L(2,q)).\]  The case $q=0$ is also proved similarly.
\end{proof}

At last in this section, we consider the cases
$m-\ell>1$.  From the results in $\S 5$, 
it is almost immediate
\begin{prop} Let $m-\ell>1$.  Then we have
\[ H^*(X_i(M(\ell,m,n))\cong
H^*(X_i(M(m-1,m,n))\cap k[y,v^{p^{m-\ell-1}}].\]
\end{prop}
From  these results, we get 
\begin{thm}
For $p\ge 5$, let $P$ be a non-abelian $p$-group of $rank_pP=2$, and  $X_i(P)$ be an irreducible component of $BP$.  Then there are graded submodules
$ H^*{(P,j)}\subset H^*(X_j(p_+^{1+2}))$ such that
\[ H^*(X_i(P))\cong  \oplus_{j\in J(i,P)}H^*{(P,j)}\]
for some index set  $J(i,P)$.
When $P$ is not a metacyclic group, we can take
$H^*{(P,j)}=H^*(X_j(p_+^{1+2}))$.
\end{thm}

{\bf Example.}    Let $p=7,q=2$.  Then we have
\[\oplus_{j=1}^{5}H^*( SY_{j,2}(E))\cong \bC\bA\{1,y,...,y^5\}\{v^2\}.\]
Hence we have
\[X_2(M(1,2,1))\cong \oplus _j^5H^*(SY_{i,2}(E))\ominus
 H^*(S(L(2,2)\vee L(2,4)))\]
\[ \cong \bC\bA\{1,y_1,y_1^3,y_1^4,y_1^5\}\{v^2\}\oplus
            \bD\bA\{y_1^2v^2\}\]
which is still given in Theorem 5.3 (letting  $S(L(2,2)\vee L(2,4))=L(2,2)$).

\section{nilpotent elements}

Let us write $H^{even}(X;\bZ)/p$ by simply $H^{ev}(X)$ so that
\[ H^{ev}(G)=H^*(G)\oplus N(G)\]
where $N(G)$ is the nilpotent ideal in $H^{ev}(G)$. 

At first, we consider metacyclic groups.  Since $BP$ is irreducible in non split cases,
we only consider in split cases.  Recall
\[P=M(\ell,m,n)=\la a,b|a^{p^m}=b^{p^n}=1, [a,b]=a^{p^{\ell}}\ra\]
for $m>\ell\ge max(m-n,1)$.  

(I) Split metacyclic groups with $\ell>m-n$.  

By  Diethelm [Dim],  its mod $p$-cohomology is
\[ H^*(P;\bZ/p)\cong k[y,u]\otimes \Lambda(x,z)\quad |y|=|u|=2,\ |x|=|z|
=1.\]
Of course all elements in $H^*(P;\bZ)$ are  (higher) $p$-torsion.  
The additive structure of $H^*(P;\bZ)/p$ is decided by that of 
$H^*(P;\bZ/p)$ by the universal coefficient theorem.
Hence we have  additively (but  not  as rings)
\[H^*(P;\bZ)/p\cong H^*(\bZ/p\times \bZ/p;\bZ)\]
\[\cong k[y,u]\{1,\beta(xz)=yz-ux\}\qquad (1).\]

The element $u\in H^2(P;\bZ/p)$ is reduced [Dim] from the spectral sequence
\[E_2^{*,*'}\cong 
H^*(P/\la a\ra; H^{*'}(\la a\ra;\bZ/p))
\Longrightarrow H^*(P;\bZ/p).\]
In fact $u=[u']\in E_{\infty}^{0,2}$ identifying $H^2(\la a\ra;\bZ/p)\cong k\{u'\}$.
Hence $u|_{\la a\ra}=u'$.  On the other hand, for the element
$v=c_{p^{m-\ell}}(\eta)$ defined in $\S 5$, $v|_{\la a\ra}=(u')^{p^{m-\ell}}$ because
the total Chern class 
in $H^*(P;\bZ/p)$ is
\[ \sum c_i(\eta)|_{\la a \ra} =(1+u')^{p^{m-\ell}}=1+(u')^{p^{m-\ell}}.\]
Therefore we see
$ v=u^{p^{m-\ell}}$ $mod(y,xz)$ in $H^*(P;\bZ/p).$

Since $H^*(P)$ is multiplicatively generated by $y$ and $v$ with $|v|\ge 2p$
from Theorem 5.1, the element $u$ is not integral class (i.e. $u\not \in Im(\rho)$
for $\rho:H^{*}(P;\bZ)\to H^*(P;\bZ/p)$).  Therefore $xz$ is an integral class
since $dim H^2(P;\bZ)/p=2$ from (1),  and 
$ H^{2}(P;\bZ/p)\cong k\{y,u,xz\}.$
Moreover we have
\begin{lemma}
The ring of the integral classes in $H^{*}(P;\bZ/p)$ is given as
\[  H^{ev}(P)\cong k[y,v]\{1, xz,xzu,...,xzu^{p^{m-\ell}-2}\}\subset H^*(P;\bZ/p).\]
\end{lemma}
\begin{proof}
Each element $y^iu^j$ is not nilpotent,
since $H^*(P;\bZ/p)\cong k[y,u]$.
Hence for $1\le j< p^{m-\ell}$, 
each element  $y^iu^j$ is not integral.
 Let $A=k[y,v]\{1,xz,...,xzu^{p^{m-\ell}-2}\}.$
Then note that 
\[ A\cong H^*(P;\bZ/p)/(\bZ/p\{y^iu^j\}|1\le j<p^{m-\ell}).\]
Hence $H^{ev}(P)\subset A$ in the lemma.

  On the other hand $dimA=n+1$ when the degree
is $2n<2p^{m-\ell}$ which is equal to $dim H^{ev}(P)$.
\end{proof} 

Let us write
\[ c_1=xz,\  c_2=xzu,\ ...,\ c_{p^{m-\ell}-1}=xzu^{p^{m-\ell}-2}. \]
Then $c_ic_j=(xz)^2u^{i+j-2}=0$.  Here recall
$H^{even}(P;\bZ)$ is multiplicatively generated by $y,c_i(\eta)$
by the argument just before Theorem 5.1 (Thomas, Huebuschmann [Th],[Hu]).  Hence we know
\begin{lemma}
We have
$ c_i=\lambda_ic_i(\eta)\  mod(y,c_1,...,c_{i-1})$ with $\lambda_i\not =0\in \bZ/p$.
\end{lemma}
\begin{proof}
By induction, assume the equation for $i-1$.
Since $c_i=xzu^{i-1}$, it  is not represented by
the polynomial of $y,c_1,...,c_{i-1}$. So it must be represented
by $c_i(\eta)$ by the result of Thomas and Huebuschmann.
\end{proof}
Thus we get 
\begin{thm}  Let $P$ be a split metacyclic group $M(\ell,m,n)$ with\\
 $\ell>m-n$.  Then we have 
\[ H^{ev}(P)\cong k[y,v]\{1,c_1,...,c_{p^{m-\ell}-1}\}\quad with\ c_ic_j=0,\]
 that is $N(P)\cong
k[y,v]\{c_1,...,c_{p^{m-\ell}-1}\}$.
\end{thm}
As $Out(P)$ modules,
$k\{c_i\}=k\{xzu^i\}\cong S_j$ when $i=j\ mod(p-1)$.  Therefore we have 
\begin{cor}  
 Let $P$ be a split metacyclic group $M(\ell,m,n)$\\
 with $\ell>m-n$.  
Then
\[H^{ev}(X_i)\cong H^*(X_i)\oplus k[y,V]
\{v^rc_s|r+s=i\ mod(p-1)\}\]
where $1\le s\le p^{m-\ell}-1.$
\end{cor} 

(II) Split metacyclic groups $P=M(\ell,m,n)$ with $\ell=m-n$.

By  also Diethelm [Dim], its mod $p$-cohomology is
\[H^*(P;\bZ/p)\cong k[y,v']\otimes\Lambda (a_1,...,a_{p-1},b,w)/(a_ia_j=a_iy=a_iw=0)\]
where $|a_i|=2i-1,|b|=1, |y|=2,|w|=2p-1, |v'|=2p$.  So we see
\[ H^*(P;\bZ/p
)/\sqrt{0}
\cong k[y,v'].
\]
Note that additively $H^*(P;\bZ)/p\cong H^*(p_-^{1+2};\bZ)/p$,
which is well known.  In particular, we get additively
\[H^{ev}(P)\cong (k[y]\oplus k\{c_1,...,c_{p-1}\})
\otimes k[v']\quad (with\ c_i=a_ib)\]
\[  \cong (k[y]\oplus k\{c_1,...,c_{p-1}\})\otimes k[v]\{1,v',...,(v')^{p^{m-\ell-1}-1}\}.\]
Therefore $H^{ev}(P)$ is additively isomorphic to
\[ H^{ev}(P)\cong \oplus_{i,j} k[v]
\{a_ib(v')^j\}
\oplus\oplus_{j} k[v,y]\{(v')^j\}\]
where $1\le i\le p-1$ and $0\le j\le p^{m-\ell-1}-1$.
Here $a_ib(v')^j$ is nilpotent and hence  integral class and let $c_{jp+i}=a_ib(v')^j$.  The element
$(v')$ is not nilpotent and we can take as the integral class $wb$ of dimension $2p$.  Let us write
$c_{pj}=wb(v')^{j-1}$.  Thus we have
\begin{thm}  Let $P$ be a split metacyclic group $M(\ell,m,n)$ with\\
 $\ell=m-n$.  Then
\[H^{ev}(P)\cong k[y,v]\oplus  k[y,v]\{c_i|i=0\ mod(p)\}\oplus k[v]\{c_i|i\not =0\ mod(p)\}\]
where $i$ ranges $1\le i\le p^{m-\ell}-1$.  Here the multiplications are given by
$c_ic_j=0$
for $0<i,j<p^{m-\ell}$ and  $yc_k=0$ for $k\not =0$ $mod(p)$.
\end{thm}
Hence we have 
\begin{cor} Let $P=M(\ell,m,n)$ for $\ell=m-n$.  Then
\[H^{ev}(X_i)\cong H^*(X_i)\oplus k[y,V]\{v^rc_s|s=0\ mod(p),\ r+s=i\ mod(p-1)\}\]
       \[ \oplus k[V]\{v^rc_s|s\not =0\ mod(p),\ r+s=i\ mod(p-1)\} .\]
\end{cor}

(III)  groups $P=C(r)$ or $ G(r',e)$. 

 Let $P=C(r)$.  Then it is known ([Ya1])
\[N(P)=\begin{cases}
            k[v]\{c_2,...,c_{p-2}\}\quad r=3,\ \ |c_i|=2i\\
            k[v]\{c_1,...,c_{p-2}\}\quad r\ge 4
\end{cases}\]
Each $c_i$ is defined as a Chern class and as $Out(C(r))$ modules, we see
$k\{c_i\}\cong det^i$.

For $P=G(r,e)$,  each $c_i$ is invariant under the action $c^*$.  Hence we have
\[N(G(r+1,e))\cong N(C(r)).\]
\begin{thm}  Let $P=C(r)$ or $G(r+1,e)$ for $r\ge 3$.  Then
\[H^{ev}(X_{j,i})\cong \begin{cases}
                     H^*(X_{0,i})\oplus k[V]\{v^rc_s|r+s=i\ mod(p-1)\}\quad j=0\\
          H^*(X_{j,i})\quad j\not =0.
\end{cases}\]
where $s$ ranges    $\begin{cases}  2\le s\le p-2 \quad for \ r=3,\\
                                    1\le s \le p-2 \quad for \ r\ge 4.
                                        \end{cases} $
\end{thm}

\section{Chow rings and motives}

For a smooth quasi projective algebraic variety $X$ over $\bC$,
let $CH^*(X)$ be the Chow ring generated by algebraic cycles of codimension $*$
modulo rational equivalence.  There is a natural (cycle) map
\[cl:  CH^*(X)\to  H^{2*}(X(\bC);\bZ).\] 
where $X(\bC)$ is the complex manifold of $\bC$-rational points of  $X$.

Let $V_n$ be  a $G$-$\bC$-vector space such that  $G$ acts  freely on $V_n-S_n$, with $codim_{V_n}S_n=n$.
Then it is known that $(V_n-S_n)/G$ is a smooth quasi-projective algebraic variety. 
Then it is known that $CH^*((V_n-S_n)/G)$ is independent 
of the choice of $V_n,S_n$.  Hence 
 Totaro defines the Chow ring of $BG$ ([To1]) by
\[ CH^*(BG)= lim_{n\to \infty}CH^*((V_n-S_n)/G).\] 
  Moreover we can approximate
$\bP^{\infty}\times BG$ by  smooth projective varieties  from Godeaux-Serre arguments ([To1]).

Let $P$ be a $p-group$.
By the Segal conjecture, the $p$-complete  automorphism $\{BP,BP\}$
of stable homotopy groups
is isomorphic to $A(P,P)_{\bZ_p}$, which is generated by transfers and map induced
from homomorphisms. 
Since $ CH^*(BP)$ also has the transfer map, we see $CH^*(BP)$ is an $A(P,
P)$-module.
For an $A(P,P)$-simple module $S$, recall $e_S$ is the corresponding 
idempotent
element and  $X_S=e_SBP$ the irreducible stable homotopy 
summand.
Let us define
\[ CH^*(X_S)=e_SCH^*(BP)\]
so that the following diagram commutes
\[ \begin{CD}
      CH^*(BP)_{(p)}  @>{cl}>> H^{2*}(BP;\bZ_{(p)}) \\
          @VVV       @VVV\\
      CH^*(X_{S})_{(p)} @>{cl}>> H^{2*}(X_{S};\bZ_{(p)}).
\end{CD}  \]

For smooth schemes $X$.$Y$ over a field $K$, let $Cor(X,Y)$ be the group 
of finite correspondences
from $X$ to $Y$ (which is a $\bZ_p$-module on the set of closed 
subvarieties of
$X\times _{K}Y$ which are finite and surjective over some connected 
component of $X$.
Let $Cor(K,\bZ_p)$ be the category of smooth schemes whose groups of 
morphisms
$Hom(X,Y)=Cor(X,Y)$.Voevodsky constructs the triangulated category $DM=DM(K,\bZ_p)$ which 
contains the category $Cor(K,\bZ_p)$ (and $limit$ of objects in $ Cor(K,\bZ_p)$).

\begin{lemma}
Let $S$ be a simple $A(P,P)$-module.  Then there  is a motive $M_S\in DM(\bC,\bZ_p)$ such that
\[ CH^*(M_S)\cong CH^*(X_S)=e_SCH^*(BP).\]
\end{lemma}
\begin{proof} Let $P$ act freely on $V-S$ so that $(V-S)/G$ approximates $BG$.
For a subgroup $i:H\subset P$, the induced map $i^*$ is defined from the projection
\[ pr : (V-S)/H\to (V-S)/G.\]
This corresponds an element in morphism in $Cor(\bC,\bZ_p)$ 
\[pr^*=\{(x,pr(x))|x\in (V-S)/H\}\in Cor((V-S)/H, (V-S)/G).\]
The transfer map  is induced from 
\[ tr=\{(pr(x),x)|x\in (V-S)/H\}\in Cor((V-S)/G,  (V-S)/H)\]
also by the definition of (finite)  correspondences.
Therefore each element in $A(P,P)$ is represented by a morphism of the 
category $DM=DM(\bC,\bZ_p)$. Moreover $DM$ is a triangulated
category and $Im(e_S)$ (i.e. the  cone of $e_S$) is an object of $DM$.
\end{proof} 

{\bf Remark.}  Of course $M_S$ is (in general) not irreducible, while $X_S$ is irreducible.

The category $Chow^{eff}(K,\bZ_p)$ of  (effective) pure Chow motives
is defined  as follows.  An object is a pair $(X,p)$ where $X$ is a projective smooth variety over $K$ and 
$p$ is a projector, i.e. $p\in Mor(X,X)$ with $p^2=p$.
Here a morphism $f\in Mor(X,Y)$ is defined as an element
$ f\in CH^{dim(Y)}(X\times Y)_{\bZ_p}.$
We say that each $M=(X,p)$ is a (pure) motive and define the Chow ring $CH^*(M)=p^*CH^*(X)$,
which is a direct summand of $CH^*(X)$.
We identify that  the motive $M(X)$ of $X$ means $(X,id.)$.
(The category $DM(K,\bZ_p)$ contains the category $Chow^{eff}(K,\bZ_p).)$ 

It is known that  we can approximate
$\bP^{\infty}\times BP$ by  smooth projective varieties  from Godeaux-Serre arguments ([To1]).
Hence we can  get the following lemma  since
\[ CH^*(X\times \bP^{\infty})\cong CH^*(X)[y]\quad |y|=1.\] 
\begin{lemma}  Let $S$ be a simple $A(P,P)$-module.
There are pure motives $M_S(i)\in Chow^{eff}(\bC,\bZ_p)$ such that
\[ lim_{i\to \infty}CH^*(M_S(i))\cong CH^*(X_S)[y],\quad deg(y)=1.\]
\end{lemma}

The following theorem is proved by Totaro, with the assumption $p\ge 5$
but without the assumption of transferred Euler classes (since
it holds when $p\ge 5$). 
\begin{thm}
(Theorem 14.3 in [To2])
Suppose $rank_pP\le 2$ and $P$ has a faithful complex representation of 
the form $W\oplus X$ where $dim(W)\le p$ and $X$ is a sum of $1$-dimensional representations. Moreover $H^{ev}(P)$ is generated by 
transferred Euler classes.
Then we have $CH^*(P)/p\cong H^{ev}(P)$.
\end{thm}
\begin{proof} (See page 179-180 in [To2].)
First note the cycle map is surjective, since $H^{ev}(P)$ is generated 
by transferred
Euler classes.
Using the Riemann-Roch without denominators,
we can show
\[ CH^*(BP)/p\cong H^{2*}(P;\bZ)/p\quad for\ *\le p.\]
By the dimensional conditions of representations $W\oplus X$ and Theorem 
12.7 in [To], we see the following map
\[ CH^*(BP)/p\to \prod_{V} CH^*(BV)\otimes_{\bZ/p} CH^{\le p-1}(BC_P(V))\]
\[ \to \prod_{V}H^*(V;\bZ/p)\otimes_{\bZ/p}H^{\le 2(p-1)}(C_P(V);\bZ/p)\]
is also injective. Here $V$ ranges elementary abelian $p$-subgroups of $P$ and
$C_P(V)$ is the centralizer group of $V$ in $P$.
So we see that the cycle map is also injective.
\end{proof}
Therefore we have
\begin{cor} Let $P$ be $C(r), G(r',e)$ or split metacyclic  groups with
$m-\ell=1$.  Then $CH^*(BP)/p\cong H^{ev}(BG)$.
\end{cor}
Totaro computed $CH^*(BP)/p$ for split metacyclic groups
with $m-\ell=1$ in 13.12 in [To]. 
 When $P$ is the extraspecial $p$-groups of order $p^3$, the above result is
proved in [Ya3].
\begin{thm}
Let $P$ be a split metacyclic $p$-group $M(\ell,m,n)$ with $m-\ell=1$,
$C(r)$ for $p\ge 3$,  or $G(r',e)$ for $p\ge 5$.
Then for each simple
 $A(P,P)$-module
 $S$, there is a motive 
$M_S\in DM(\bC,\bZ_p)$ with
\[ CH^*(M_S)/p\cong H^{ev}(X_S)=H^{even}(X_S;\bZ)/p.\]
\end{thm}

For a cohomology theory $h^*(-)$, define the $h^*(-)$-theory
topological nilpotence degree $d_0(h^*(BG))$ to be the least nonnegative integer $d$ such that
the map
\[ h^*(BG)/p \to \prod_{
V: el.ab.}h^*(BV)\otimes h^{\le d}(BC_G(V))/p\]      
(where $V$ ranges elementary abelian $p$-subgroups of $G$) is injective.  Note that $d_0(H^*(BG;\bZ))\le d_0(H^*(BG;\bZ/p))$.

Totarto computed it  in the many  cases of groups $P$ with $rank_pP=2$.
In particular, if $P$ is a split metacyclic $p$-group for $p\ge 3$,  then $d_0(H^*(P;\bZ/p))=2$ and
$d_0(CH^*(BP))=1$ when $m-\ell=1$.   Hence $d_0(H^*(P;\bZ))=2$ for these
 split metacyclic groups $P$ (for $p\ge 3$).

This fact also is shown directly from Theorem 9.3 and 9.5. 
Let $P=M(\ell,m,n)$ with $m-\ell=1$.
 Consider the restriction map
\[ H^{ev}(P)\to H^{ev}(V)\otimes
 H^2(P;\bZ)/p\quad ( where\ V=\la a^{p^{m-1}}\ra\subset  Z(P): center)\]
induced the product map $V\times P \to P$.  
Let $\ell=m-1>m-n$.
Then the element is defined in  Lemma 9.1, 9.2
   \[  c_j=xzu^{j-1}\mapsto \sum_{i} u^{j-i-1}\otimes xzu^i
 \equiv u^{j-1}\otimes c_1\not =0 \in  H^{ev}(V)\otimes H^2(P;\bZ)/p.\]
 For $\ell=m-n$, we can see
$d_0(H^*(P;\bZ))=2$ similarly.
% For $\ell=m-n$ and $n=1$, we  also see that the nilpotent %element . $x_j$ maps to
%$ab\otimes u^{j-1}$ (or $wb\otimes u^{j-p-1}$ for $j=0\ mod %(p))$ in $H^{ev}(V)\otimes H^2(P) $.    
%(From the proof of Theorem 2 in [Dim], we see $w|%V=zu^{p-1}$.)   

\end{document}